\documentclass{amsart}

\usepackage{graphicx}

\usepackage{amssymb}

\usepackage[shortlabels]{enumitem}
\usepackage{color}
\usepackage{tikz}
\usetikzlibrary{matrix}
\usepackage{hyperref}
\usepackage {cleveref}

\newtheorem{theorem}{Theorem}[section]
\newtheorem{proposition}[theorem]{Proposition}
\newtheorem{lemma}[theorem]{Lemma}
\newtheorem{corollary}[theorem]{Corollary}

\theoremstyle{definition}
\newtheorem{definition}[theorem]{Definition}

\theoremstyle{remark}
\newtheorem{remark}[theorem]{Remark}
\newtheorem{example}[theorem]{Example}

\numberwithin{equation}{section}

\begin{document}

\title[Singular Patterns of Generic Maps]{Singular Patterns of Generic Maps \\
of Surfaces with Boundary into the Plane}

\author{Dominik J. Wrazidlo}
\address{Institute of Mathematics for Industry, Kyushu University, Motooka 744, Nishi-ku, Fukuoka 819-0395, Japan}
\email{d-wrazidlo@imi.kyushu-u.ac.jp}

\date{\today}


\begin{abstract}
For generic maps from compact surfaces with boundary into the plane we develop an explicit algorithm for minimizing both the number of cusps and the number of components of the singular locus.
More precisely, we minimize among maps with fixed boundary conditions and prescribed \emph{singular pattern}, by which we mean the combinatorial information of how the $1$-dimensional singular locus meets the boundary.
Each step of our algorithm modifies the given map only locally by either creating or eliminating a pair of cusps.
We show that the number of cusps is an invariant modulo $2$ and can be reduced to at most one, and we compute the minimal number of components of the singular locus in terms of the prescribed data.
Applications include a discussion of pseudo-immersions as well as the computation of state sums in Banagl's positive topological field theory.
\end{abstract}

\maketitle

\section{Introduction}

By a classical result of Whitney \cite{whi}, any smooth map $\mathbb{R}^{2} \rightarrow \mathbb{R}^{2}$ (``smooth'' always means differentiable of class $C^{\infty}$) can be approximated arbitrarily well by a smooth map $F \colon \mathbb{R}^{2} \rightarrow \mathbb{R}^{2}$ such that for every point $p \in \mathbb{R}^{2}$, the map germ of $F$ at $p$ is smoothly right-left equivalent to one of the following map germs $(\mathbb{R}^{2}, 0) \rightarrow (\mathbb{R}^{2}, 0)$:
\begin{align*}
(x, y) \mapsto \begin{cases}
(x, y), \quad &\text{$p$ is a regular point of } F, \\
(x, y^{2}), \quad &\text{$p$ is a fold point of } F, \\
(x, xy + y^{3}), \quad &\text{$p$ is a cusp of } F.
\end{cases}
\end{align*}
In general, a smooth map $F \colon W \rightarrow V$ between smooth $2$-manifolds $W$ and $V$ is called \emph{generic} if its singular locus
$$S(F) = \{p \in W; \; \text{the differential } dF_{p} \colon T_{p}W \rightarrow T_{F(p)}V \text{ has rank } < 2\}$$
consists only of fold points and cusps.
If $F \colon W \rightarrow V$ is generic, then $S(F) \subset W$ is a $1$-dimensional submanifold, and the cusps of $F$ form a discrete subset of $W$.
Moreover, a generic map $F \colon W \rightarrow V$ is called fold map if $S(F)$ contains no cusps.
Note that any fold map $F \colon W \rightarrow V$ restricts to an immersion $S(F) \rightarrow V$.

The problem of classifying generic maps between surfaces in terms of suitable notions of complexity related to their singular locus has attracted the attention of several authors; see e.g. \cite{kal, myam} for counting cusps and components of the singular locus, \cite{tyama} for the study of apparent contours of stable maps (compare \Cref{remark apparent contour}), and \cite{hac} for the combinatorics of certain graphs associated to stable maps.

Our paper contributes to this field of study by determining the possible number of cusps and of singular components (i.e., components of the singular locus) for generic maps of compact surfaces with boundary into the plane.
For the generic maps under consideration, we impose reasonable boundary conditions, and also fix their \emph{singular pattern} (see \Cref{pattern}), namely the combinatorial information of how the $1$-dimensional singular locus runs into the boundary.

Initial motivation for our setting comes from the requirement that the number of singular components of a generic map should be detectable in a way that behaves nicely under gluing of cobordisms -- a property that is desirable from the viewpoint of topological field theory (TFT).
In fact, suppose that a $2$-dimensional cobordism $W$ is the result of gluing a cobordism $W_{1}$ from $M$ to $N$ and a cobordism $W_{2}$ from $N$ to $P$ along their common boundary part $N$, i.e., $W = W_{1} \cup_{N} W_{2}$.
Consider a generic map $F \colon W \rightarrow \mathbb{R}^{2}$ whose singular locus $S(F)$ is transverse to $M$, $N$ and $P$.
Then, the number of singular components of $F$ cannot directly be computed from the number of singular components of $F|_{W_{1}}$ and $F|_{W_{2}}$.
The remedy is to record also the additional information that is provided by singular patterns.
It is exactly this gluing principle that has been exploited by Banagl \cite[Section 10]{ban} in order to construct an explicit example of a so-called positive TFT which is based on certain fold maps of cobordisms into the plane (see \Cref{remark positive TFT}).
The spirit of TFT is also visible in the proof of \Cref{proposition boundary turning invariant and Euler characteristic}, where a relative version of the winding number for arcs in the plane is used.

Let us introduce some terminology to prepare the discussion of our main results.
Suppose that $W$ is a connected compact smooth $2$-manifold with boundary $P = \partial W$, and fix a collar $[0, \infty) \times P \subset W$.
A \emph{boundary condition} is a fold map $f \colon (-\varepsilon, \varepsilon) \times P \rightarrow \mathbb{R}^{2}$ whose singular locus $S(f)$ is transverse to $\{0\} \times P$.
By a \emph{(singular) pattern} (see \Cref{pattern}) on $W$ we mean a pair $(f, \varphi)$ consisting of a boundary condition $f$, and a partition $\varphi$ of the finite set $P_{f} := (\{0\} \times P) \cap S(f)$ into subsets of cardinality $2$.
Finally, a \emph{realization} of a pattern $(f, \varphi)$ is a generic extension $F \colon W \rightarrow \mathbb{R}^{2}$ of the germ of $f|_{[0, \varepsilon) \times P}$ at $\{0\} \times P$ (such $F$ do always exist) such that a subset $A \subset P_{f}$ belongs to the partition $\varphi$ if and only if $A$ is the boundary of a component of $S(F)$.
In \Cref{proposition existence of realizations} we will provide a combinatorial criterion for deciding whether a realization of a given pattern exists.

The proofs of our main results will provide an explicit algorithm for minimizing the number of cusps (\Cref{MAIN THEOREM 1}) and the number of components (\Cref{MAIN THEOREM 2}) of realizations of a given pattern on $W$.
Our algorithm relies merely on two local modifications of generic maps which can be realized by homotopies supported in $W \setminus \partial W$.
Namely, we use Levine's homotopy \cite{lev} for eliminating pairs of cusps of generic maps between surfaces, as well as the complementary process of creating a pair of cusps on a fold line by means of the swallow-tail homotopy.
These two modifications, say (E) and (C), serve as elementary building blocks for more complicated moves (see \Cref{local moves}) that will be used to build up our algorithm.
Note that the moves (E) and (C) correspond to two out of $10$ types of codimension one local strata of the discriminant hypersurface $\Gamma \subset C^{\infty}(W, \mathbb{R}^{2})$ consisting of smooth maps $W \rightarrow \mathbb{R}^{2}$ that have non-generic apparent contours \cite{aic}.

\begin{theorem}\label{MAIN THEOREM 1}
Let $(f, \varphi)$ be a pattern on $W$.
\begin{enumerate}[(a)]
\item All realizations of $(f, \varphi)$ have the same number of cusps modulo $2$.
\item Every realization of $(f, \varphi)$ can be modified on $W \setminus \partial W$ by a finite sequence of (E) and (C) moves to obtain a realization with at most one cusp.
\end{enumerate}
\end{theorem}

\newpage
In particular, part (a) yields a $\mathbb{Z}/2$-valued invariant of certain patterns on $W$.
This invariant turns out to be independent of $\varphi$, and can be considered as a relative Thom polynomial for cusps of generic maps of surfaces into the plane.
In \Cref{proposition parity of number of cusps of realizations} we provide a formula for computing this invariant in terms of $f$.
It involves the \emph{boundary turning invariant} of $f$, which will be defined in \Cref{Boundary Turning Invariant} by means of winding numbers of certain immersions of the circle into the plane.
Our formula resembles the formula of Fukuda-Ishikawa \cite{fukishi} for maps between compact surfaces with boundary (see \Cref{remark Fukuda-Ishikawa theorem}).
In the context of \Cref{proposition parity of number of cusps of realizations} we define an integer $\Gamma^{\sigma}$ that depends on an orientation $\sigma$ of $P$.
We will use $\Gamma^{\sigma}$ in the definition of another integer $\Delta^{\sigma}$ (see \Cref{Realizations and Number of Loops}) that will appear in our second main result concerning the minimal number of loops (i.e., singular components diffeomorphic to the circle) that can occur for realizations of a given pattern on $W$.

\begin{theorem}\label{MAIN THEOREM 2}
Let $(f, \varphi)$ be a pattern on $W$ such that $P_{f} \neq \emptyset$.
Suppose that $F \colon W \rightarrow \mathbb{R}^{2}$ is a realization of $(f, \varphi)$ without cusps.
\begin{enumerate}[(a)]
\item Let $W$ be non-orientable.
Then, given an integer $l \geq 0$, $F$ can be modified on $W \setminus \partial W$ by a finite sequence of (E) and (C) moves to a realization of $(f, \varphi)$ which has no cusps and $l$ loops.
\item Let $W$ be orientable, and let $\sigma$ denote an orientation of $P = \partial W$ which is induced by one of the two orientations of $W$.
Then, the following statements are equivalent for any integer $l$:
\begin{enumerate}[(i)]
\item $F$ can be modified on $W \setminus \partial W$ by a finite sequence of (E) and (C) moves to a realization of $(f, \varphi)$ which has no cusps and $l$ loops.
\item $l \in \mathbb{N} \cap (\Delta^{\sigma} + 2 \mathbb{N}) \cap (\Delta^{-\sigma} + 2 \mathbb{N})$. 
\end{enumerate}
\end{enumerate}
\end{theorem}

In \Cref{applications} we use the gluing principle for singular patterns to generalize \Cref{MAIN THEOREM 2} to the case $P_{f} = \emptyset$ of pseudo-immersions (see \Cref{example pseudo-immersions}), as well as to the case that $F$ has cusps (see \Cref{corollary extending main result to cusps}).

The paper is structured as follows.
In \Cref{local moves} we present the local moves that will be essential for our approach.
The concept of boundary turning invariant will be introduced in \Cref{Boundary Turning Invariant}.
The proofs of our main results will cover \Cref{Number of Cusps of Generic Extensions} to \Cref{proof of main theorem 2}.
\Cref{applications} concludes with some applications.

\subsection*{Notation}
Let $\mathbb{N} = \{0, 1, 2, \dots\}$ denote the set of non-negative integers.
Let $W$ denote a $2$-dimensional connected compact smooth manifold.
In case of a non-empty boundary $P = \partial W$ we also fix a collar neighborhood $[0, \infty) \times P \subset W$.
Let $\chi(W)$ denote the Euler characteristic of $W$.
The plane $\mathbb{R}^{2}$ is equipped with the standard orientation.
We use the following convention for the induced orientation on the boundary of an oriented surface.
If the first vector of an oriented frame points out of the surface at a boundary point, then the second vector defines the orientation of the boundary at that point.

\subsection*{Acknowledgements}
The material presented in this paper is based on Chapter 5 of the author's PhD thesis \cite{wra2}.
The author is grateful to his advisor Professor Markus Banagl for constant encouragement and many fruitful discussions.
The author would like to thank Professor Takahiro Yamamoto, Professor Minoru Yamamoto and Professor Osamu Saeki for discussions and invaluable comments on an early version of the manuscript.
Also, the author thanks the anonymous referee for reading the paper very carefully, which led to several improvements.

The author is grateful to the German National Merit Foundation (Studienstiftung des deutschen Volkes) for financial support.
Moreover, the author has been supported by JSPS KAKENHI Grant Number JP17H06128.

\section{Local modifications}\label{local moves}

Throughout this section, let $F \colon W \rightarrow \mathbb{R}^{2}$ denote a realization of some pattern $(f, \varphi)$ on $W$ as defined in the introduction (see also \Cref{pattern}).

The discussion in this section will be supported by figures showing the singular locus of $F$ in open neighborhoods $U \subset W \setminus \partial W$.
These neighborhoods $U$ will be colored in grey.
The intersection $S(F) \cap U$ will be represented by massive black lines, whereas $S(F) \cap (W \setminus U)$ will be indicated by spotted lines in the figures.
The cusps of $F$ will be marked by discrete points on the massive lines.
Every cusp $c$ is furthermore equipped with a vector in $T_{x}W$ that points \emph{downward} in the sense of Levine (see the definition on p. 284 of \cite{lev}).
(A vector points downward at a cusp if its image in the plane points in the direction of the cusp.
Intuitively, the vector itself points into the half plane into which the cusp can propagate.)
The direction of a downward pointing tangent vector will be referred to as the direction into which the cusp \emph{points}.
Finally, \emph{joining curves} (see Section (4.4) in \cite[p. 285]{lev}) between cusps will be symbolized by dashed lines (see the left of \Cref{elementary moves}).
(A joining curve between two cusps $c_{0}$ and $c_{1}$ is an embedding $\alpha \colon [0, 1] \rightarrow W \setminus \partial W$ such that $\alpha^{-1}(S(F)) = \{0, 1\}$, $\alpha(i) = c_{i}$, $i = 0, 1$, and $(-1)^{i} \cdot \alpha'(i) \in T_{c_{i}}W$ points downward at $c_{i}$.)

Our approach is based on the following two local modifications (see \Cref{elementary moves}).

\begin{itemize}[(E)]
\item \textbf{Elimination of cusp pairs} (see \Cref{elementary moves}(a)).
A pair of cusps of $F$ can be eliminated if the cusps point into a common component of $W \setminus S(F)$ (see Fig. 3 in \cite[p. 286]{lev}).
In fact, it follows from the proofs of Lemma (1) and (2) in Section (4.4) of \cite[p. 285]{lev} that such a pair of cusps can be connected by a joining curve.
As $W$ is a $2$-dimensional cobordism, the pair of cusps of $F$ is a matching pair in the sense of \cite{lev} and hence \emph{removable} (as defined in Section (4.5) of \cite[p. 285]{lev}).
As explained in \cite[p. 286]{lev}, the pair of cusps can therefore be eliminated by a homotopy of $F$ that modifies $F$ only in a tubular neighborhood of the joining curve.
\end{itemize}

\begin{itemize}[(C)]
\item \textbf{Creation of cusp pairs} (see \Cref{elementary moves}(b)).
In a small neighborhood of a fold point of $F$ a pair of two new cusps can be introduced on the fold line in such a way that the cusps point into opposite directions of the singular curve.
This modification can be realized by a homotopy of $F$ that modifies $F$ only in this neighborhood and is based on the swallow-tail homotopy (for details, see for instance Section 4.7 in \cite[p. 110]{wra2}).
\end{itemize}

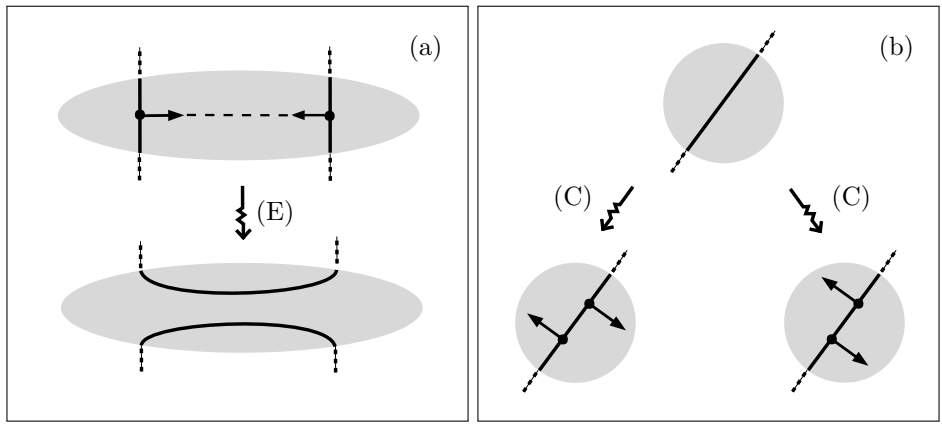
\begin{figure}[htbp]
\centering
\fbox{\begin{tikzpicture}
    \draw (0, 0) node {\includegraphics[height=0.4\textwidth]{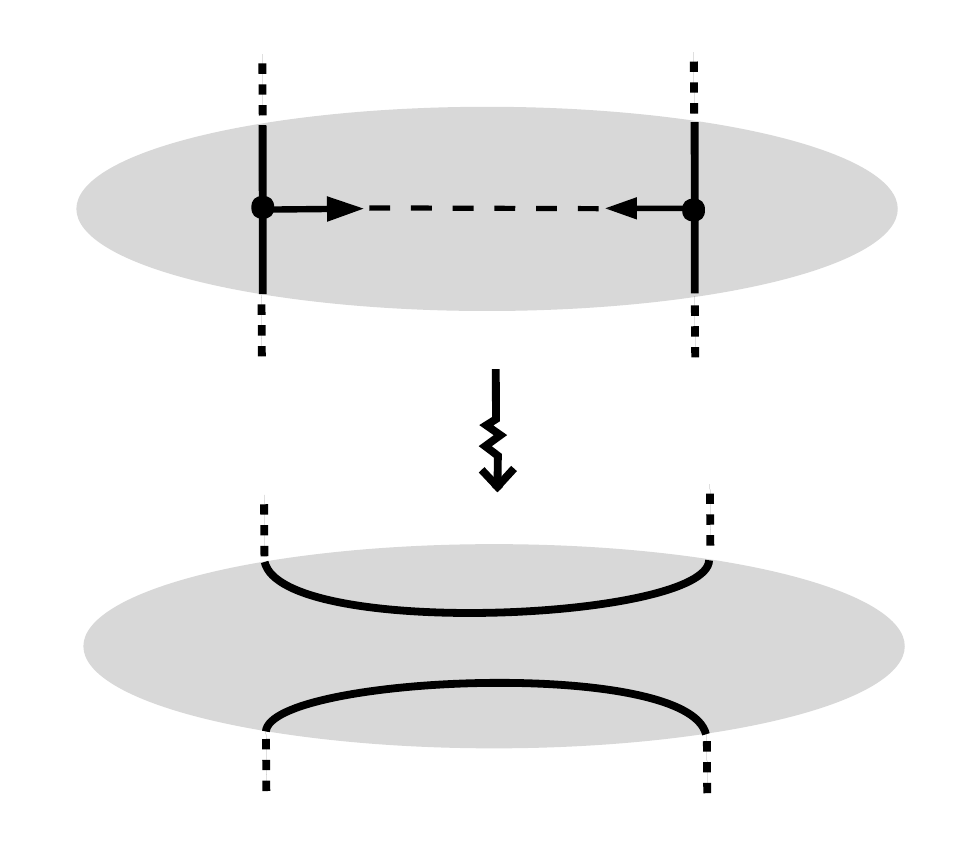}};
    \draw (0.5, 0) node {(E)};
    \draw (2.5, 2.2) node {(a)};
\end{tikzpicture}}
\fbox{\begin{tikzpicture}
    \draw (0, 0) node {\includegraphics[height=0.4\textwidth]{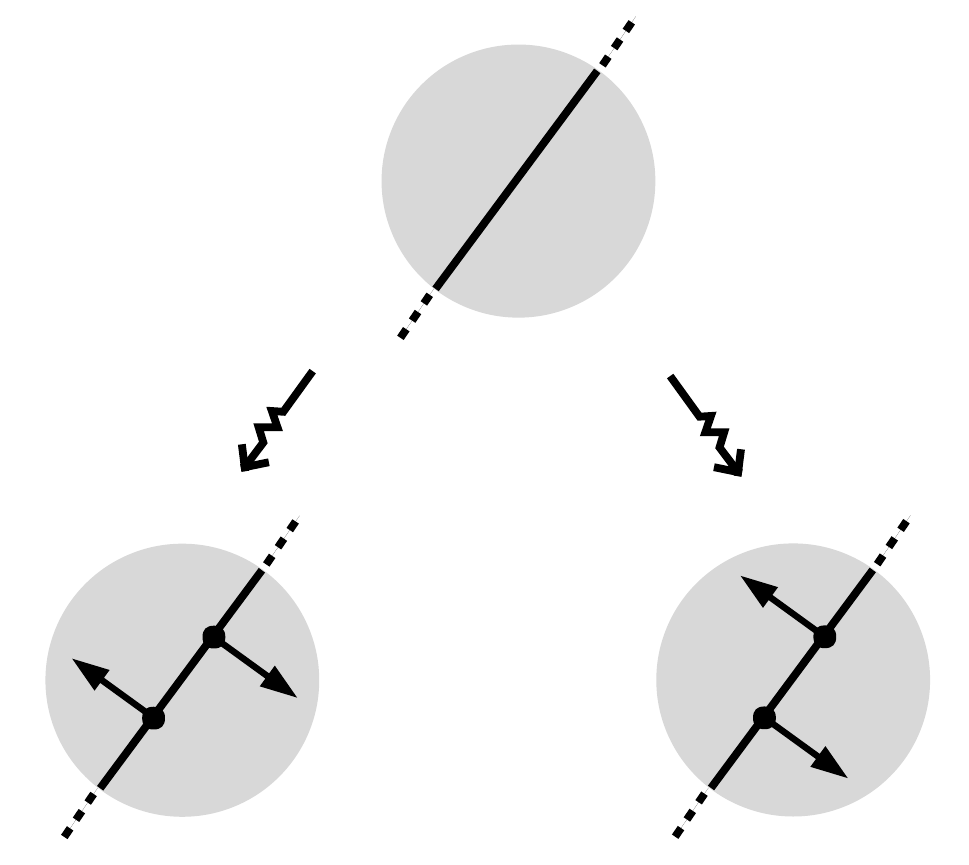}};
    \draw (1.9, 0.2) node {(C)};
    \draw (-1.8, 0.2) node {(C)};
    \draw (2.5, 2.2) node {(b)};
\end{tikzpicture}}
\caption{(a): Levine's homotopy \cite{lev} for \textbf{eliminating pairs of cusps} of generic maps between surfaces.
(b): The swallow-tail homotopy for \textbf{creating a pair of cusps} on a fold line.}
\label{elementary moves}
\end{figure}

\begin{definition}\label{definition loops}
By a \emph{loop} of $F$ we mean a component of $S(F)$ that is diffeomorphic to the circle.
A loop $C$ of $F$ is called
\begin{itemize}
\item \emph{pure} if $C$ contains no cusps of $F$.
\item \emph{trivial} if the normal bundle of $C$ in $W$ is trivial.
\item \emph{contractible} if there is an open disc $D \subset W \setminus \partial W$ such that $D \cap S(F) = C$.
\end{itemize}
\end{definition}

\begin{proposition}[Loop generation]\label{lemma new fold loops}
\begin{enumerate}[(i)]
\item Let $p \in W \setminus \partial W$ be a fold point of $F$.
Given a neighborhood $U \subset W$ of $p$, $F$ can be modified on $U$ by two (C) moves followed by two (E) moves in such a way that the modified map $\widetilde{F} \colon W \rightarrow \mathbb{R}^{2}$ has two loops more than $F$.
\item Let $c_{0}, c_{1}$ be two cusps of $F$.
Suppose that $\alpha \colon [0, 1] \rightarrow W \setminus \partial W$ is a joining curve (see Section (4.4) in \cite[p. 285]{lev}) between $\alpha(0) = c_{0}$ and $\alpha(1) = c_{1}$.
Fix a neighborhood $V$ of $\alpha([0, 1])$.
Let $\gamma_{i} \colon [-1, 1] \rightarrow S(F) \cap V$, $i = 0, 1$, be embeddings such that $\gamma_{i}(0) = c_{i}$ and $\gamma_{0}([-1, 1]) \cap \gamma_{1}([-1, 1]) = \emptyset$.
Suppose that $U \subset V$ is a neighborhood of $\alpha([0, 1])$ which does not contain the points $\gamma_{i}(\pm 1)$, $i = 0, 1$.
Then, $F$ can be modified on $U$ by one (E) move along the joining curve $\alpha$, followed by two (C) moves, followed by two (E) moves in such a way that the modified map $\widetilde{F} \colon W \rightarrow \mathbb{R}^{2}$ has one loop more than $F$ (namely, a contractible loop on $U$), and there exist embeddings $\widetilde{\gamma}_{i} \colon [-1, 1] \rightarrow S(\widetilde{F}) \cap V$, $i = 0, 1$, such that $\widetilde{\gamma}_{i}(j) = \gamma_{i}(j)$ for $j = \pm 1$.
\end{enumerate}
\end{proposition}

\begin{definition}\label{definition canonical orientation}
The \emph{canonical orientation} of $S(F)$ is uniquely determined by requiring that, at every fold point of $F$, $W$ is ``folded to the left'' into the plane, by which we mean the following.
Recall that every fold point $p \in S(F)$ admits a neighborhood $U \subset W$ such that $F$ restricts to an embedding $S(F) \cap U \rightarrow \mathbb{R}^{2}$, and $F(U) \subset \mathbb{R}^{2}$ is a codimension $0$ submanifold with boundary $F(S(F) \cap U)$.
Then the canonical orientation at $p \in S(F)$ is such that the induced orientation at $F(p) \in F(S(F) \cap U)$ coincides with the orientation of $F(S(F) \cap U)$ induced by $F(U) \subset \mathbb{R}^{2}$.
Note that the resulting orientation of the fold locus of $F$ extends unambiguously over the cusps of $F$.
\end{definition}

\begin{proposition}[General elimination of cusps]\label{lemma generalized cusp elimination}
Let $c_{0}, c_{1}$ be two cusps of $F$.
Suppose that $\alpha \colon [0, 1] \rightarrow W \setminus \partial W$ is an embedding such that $\alpha(i) = c_{i}$, $i = 0, 1$, $(-1)^{i} \cdot \alpha'(i)$ is a downward pointing tangent vector of $c_{i}$, $i = 0, 1$, and $\alpha|_{(0, 1)}$ intersects $S(F)$ transversely and only in fold points of $F$.
Fix a neighborhood $U$ of $\alpha([0, 1])$ and orientation preserving embeddings $\gamma_{0}, \gamma_{1} \colon [0, 1] \rightarrow S(F)$ (where $S(F)$ is equipped with the canonical orientation of \Cref{definition canonical orientation}) such that $\gamma_{0}(0) \notin U$, $\gamma_{0}(1) = c_{0}$, and $\gamma_{1}(0) = c_{1}$, $\gamma_{1}(1) \notin U$.
Then, $F$ can be modified on $U$ by a finite sequence of (E) and (C) moves in such a way that the modified map $\widetilde{F} \colon W \rightarrow \mathbb{R}^{2}$ has two cusps less than $F$, and there exists an embedding $\gamma \colon [0, 1] \rightarrow S(\widetilde{F})$ such that $\gamma(i) = \gamma_{i}(i)$, $i = 0, 1$.
\end{proposition}

\begin{proof}
We proceed by induction over the cardinality of $\alpha^{-1}(S(F)) = \{t_{0}, \dots, t_{r}\}$, where $0 = t_{r} < \dots < t_{0} = 1$.
If $r = 1$, then the claim follows by applying (E) to the joining curve $\alpha$, where note that the existence of $\gamma$ is automatic.
If $r > 1$, then apply one of the local modifications indicated in \Cref{induction step} depending on the position of $\gamma_{0}(0)$, and modify $\alpha$ on $[0, t_{r-2})$ by reducing the number of cusps by one.
\end{proof}

\begin{figure}[htbp]
\centering
\fbox{\begin{tikzpicture}
    \draw (0, 0) node {\includegraphics[height=0.4\textwidth]{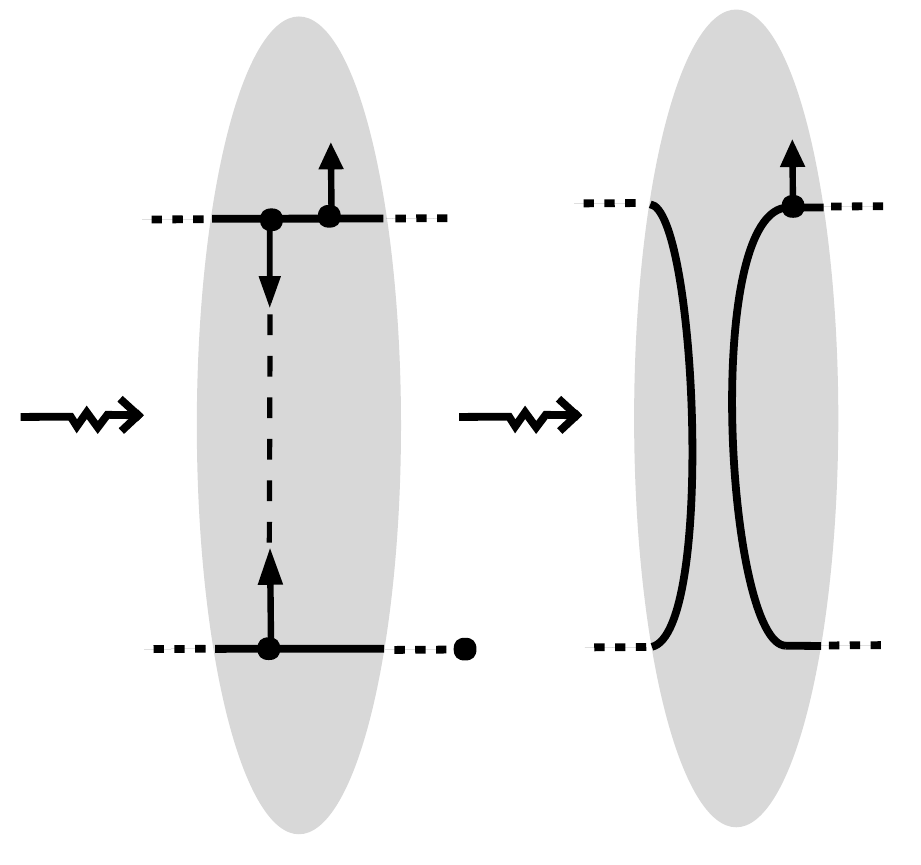}};
    \draw (0.3, 0.4) node {(E)};
    \draw (-2.2, 0.4) node {(C)};
    \draw (-1.0, -1.5) node {$c_{0}$};
    \draw (0, -1.6) node {$\gamma_{0}(0)$};
\end{tikzpicture}}
\fbox{\begin{tikzpicture}
    \draw (0, 0) node {\includegraphics[height=0.4\textwidth]{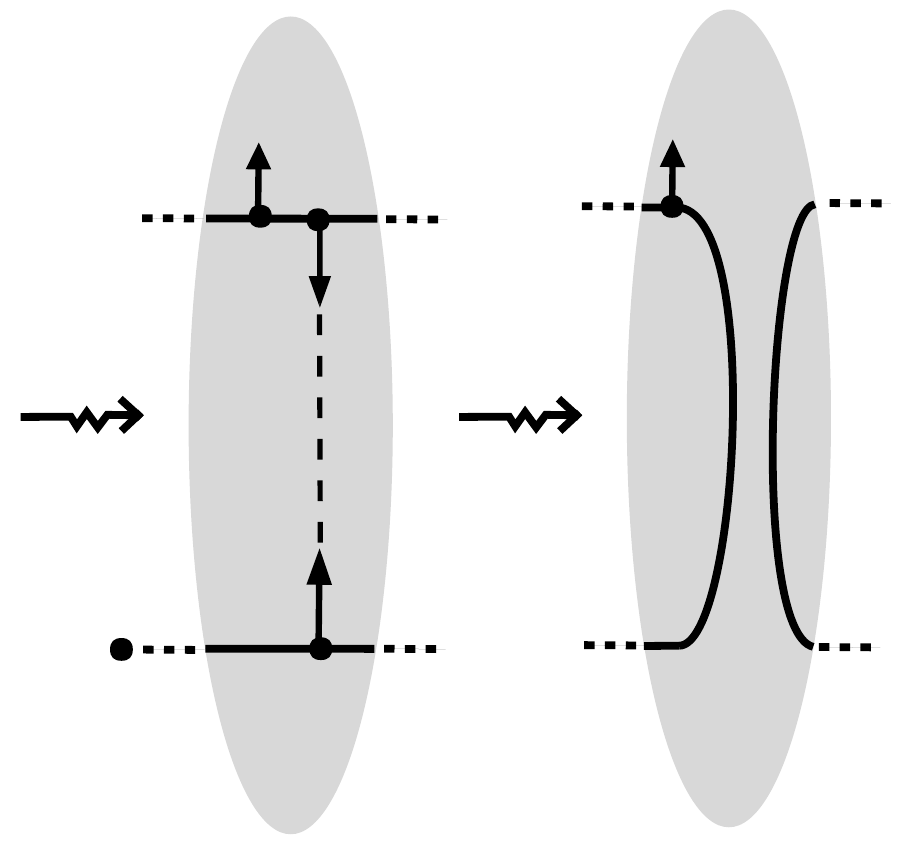}};
    \draw (0.3, 0.4) node {(E)};
    \draw (-2.2, 0.4) node {(C)};
    \draw (-0.8, -1.5) node {$c_{0}$};
    \draw (-2.1, -1.6) node {$\gamma_{0}(0)$};
\end{tikzpicture}}
\caption{Induction step for general elimination of cusps.}
\label{induction step}
\end{figure}
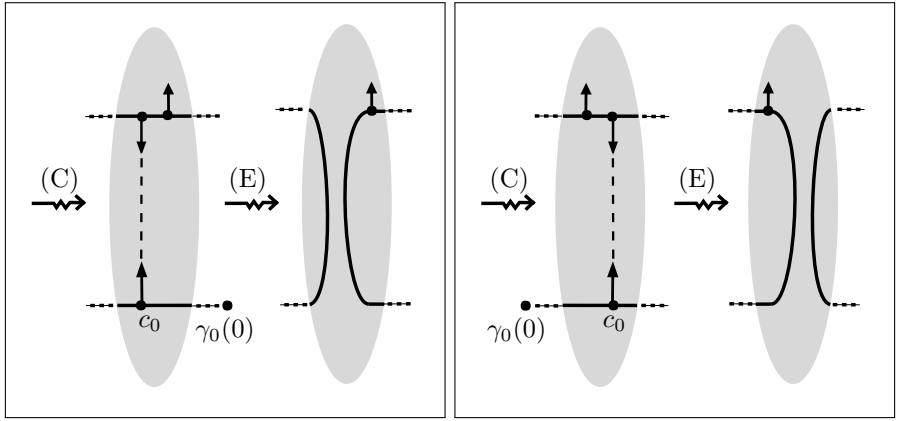

\begin{remark}
An explicit version of \Cref{lemma generalized cusp elimination} (not claiming the existence of the embedding $\gamma$) is originally due to K\'{a}lm\'{a}n (see Lemma 1.4 of \cite[p. 308]{kal}).
\end{remark}

\begin{proposition}[Loop simplification, see \Cref{figure loop simplification}]\label{proposition loop simplification}
Let $C$ be a loop of $F$.
Suppose that $\alpha \colon [0, 1] \rightarrow W \setminus \partial W$ is an embedding such that $\alpha^{-1}(S(F)) = \{0, 1\}$, where $\alpha(0) \in C$ and $\alpha(1) \notin C$ are fold points of $F$.
Fix a neighborhood $U$ of $\alpha([0, 1])$.
Then, $F$ can be modified on $U$ by a finite sequence of (E) and (C) moves in such a way that the modified map $\widetilde{F} \colon W \rightarrow \mathbb{R}^{2}$ has a contractible loop on $U$, and $F$ and $\widetilde{F}$ have the same number of loops and the same number of cusps.
\end{proposition}

\begin{proposition}[Loop reduction, see \Cref{figure loop reduction}]\label{proposition loop reduction}
Let $C_{0}, C_{1}$ be two loops of $F$.
Suppose that $\alpha \colon [0, 1] \rightarrow W \setminus \partial W$ is an embedding such that $\alpha^{-1}(S(F)) = \{0, t, 1\}$ for some $t \in (0, 1)$, where $\alpha(i) \in C_{i}$, $i = 0, 1$, are fold points of $F$, and $\alpha$ intersects $S(F)$ transversely at $t$ in a fold point $\alpha(t) \notin C_{0} \cup C_{1}$ of $F$.
Fix a neighborhood $U$ of $\alpha([0, 1])$.
Then, $F$ can be modified by a finite sequence of (E) and (C) moves in such a way that the modified map $\widetilde{F} \colon W \rightarrow \mathbb{R}^{2}$ has two loops less than $F$, and $F$ and $\widetilde{F}$ have the same number of cusps.
\end{proposition}

\begin{proposition}[Tunneling, see \Cref{figure tunneling}]\label{proposition tunneling}
Let $C$ be a trivial loop of $F$.
Suppose that $\alpha \colon [0, 1] \rightarrow W \setminus \partial W$ is an embedding such that $\alpha^{-1}(S(F)) = \{0, t_{0}, t_{1}\}$ for some $0 < t_{0} < t_{1} < 1$, where $\alpha(0) \in C$, and $\alpha$ intersects $S(F)$ at $t_{i}$, $i = 0, 1$, transversely in fold points $\alpha(t_{i}) \notin C$ of $F$.
Fix a neighborhood $V$ of $C \cup \alpha([0, 1])$.
Let $C' \subset V \setminus S(F)$ be an embedded circle which is contained in a tubular neighborhood of $C \subset W$ such that $\alpha^{-1}(C') = \{s\}$ for some $0 < s < t_{0}$.
Let $\gamma_{i} \colon [-1, 1] \rightarrow S(F) \cap V$, $i = 0, 1$, be embeddings such that $\gamma_{i}(0) = \alpha(t_{i})$ and $\gamma_{0}([-1, 1]) \cap \gamma_{1}([-1, 1]) = \emptyset$.
Suppose that $U \subset V$ is a neighborhood of $C' \cup \alpha([s, t_{1}])$ which does neither contain $C$ nor any of the points $\alpha(1)$ and $\gamma_{i}(\pm 1)$, $i = 0, 1$.
Then, $F$ can be modified on $U$ by a finite sequence of (E) and (C) moves in such a way that the modified map $\widetilde{F} \colon W \rightarrow \mathbb{R}^{2}$ has the same number of loops and the same number of cusps as $F$, and there exists an embedding $\widetilde{\alpha} \colon [0, 1] \rightarrow V$ such that $\widetilde{\alpha}^{-1}(S(\widetilde{F})) = \{0\}$ and $\widetilde{\alpha}(i) = \alpha(i)$, $i = 0, 1$, and embeddings $\widetilde{\gamma}_{i} \colon [-1, 1] \rightarrow S(\widetilde{F}) \cap V$, $i = 0, 1$, such that $\widetilde{\gamma}_{i}(j) = \gamma_{i}(j)$ for $j = 0, 1$.
\end{proposition}

\begin{proposition}[Balancing, see \Cref{figure balancing}]\label{proposition balancing}
Let $C_{0}, C_{1}$ be two loops of $F$.
Suppose that $\alpha \colon [0, 1] \rightarrow W \setminus \partial W$ is an embedding such that $\alpha^{-1}(S(F)) = \{0, t_{0}, t_{1}, 1\}$ for some $0 < t_{0} < t_{1} < 1$, where $\alpha(i) \in C_{i}$, $i = 0, 1$, are fold points of $F$, and $\alpha$ intersects $S(F)$ at $t_{i}$, $i = 0, 1$, transversely in fold points $\alpha(t_{i}) \notin C_{0} \cup C_{1}$ of $F$.
Furthermore, suppose that $H \subset W \setminus S(F)$ is an embedded circle that intersects $\alpha$ transversely, and such that $\alpha^{-1}(H) = \{s\}$ for some $s \in (t_{0}, t_{1})$.
Fix a neighborhood $V$ of $H \cup \alpha([0, 1])$.
Let $\gamma_{i} \colon [-1, 1] \rightarrow S(F) \cap V$, $i = 0, 1$, be embeddings such that $\gamma_{i}(0) = \alpha(t_{i})$ and $\gamma_{0}([-1, 1]) \cap \gamma_{1}([-1, 1]) = \emptyset$.
Suppose that $U \subset V$ is a neighborhood of $H \cup \alpha([0, 1])$ which does not contain the points $\gamma_{i}(\pm 1)$, $i = 0, 1$.
Then, $F$ can be modified on $U$ by a finite sequence of (E) and (C) moves in such a way that the modified map $\widetilde{F} \colon W \rightarrow \mathbb{R}^{2}$ has two loops less than $F$, whereas $F$ and $\widetilde{F}$ have the same number of cusps, and there exist embeddings $\widetilde{\gamma}_{i} \colon [-1, 1] \rightarrow S(\widetilde{F}) \cap V$, $i = 0, 1$, such that $\widetilde{\gamma}_{i}(j) = \gamma_{i}(j)$ for $j = 0, 1$.
\end{proposition}

\begin{figure}[htbp]
\centering
\fbox{\begin{tikzpicture}
    \draw (0, 0) node {\includegraphics[height=0.24\textwidth]{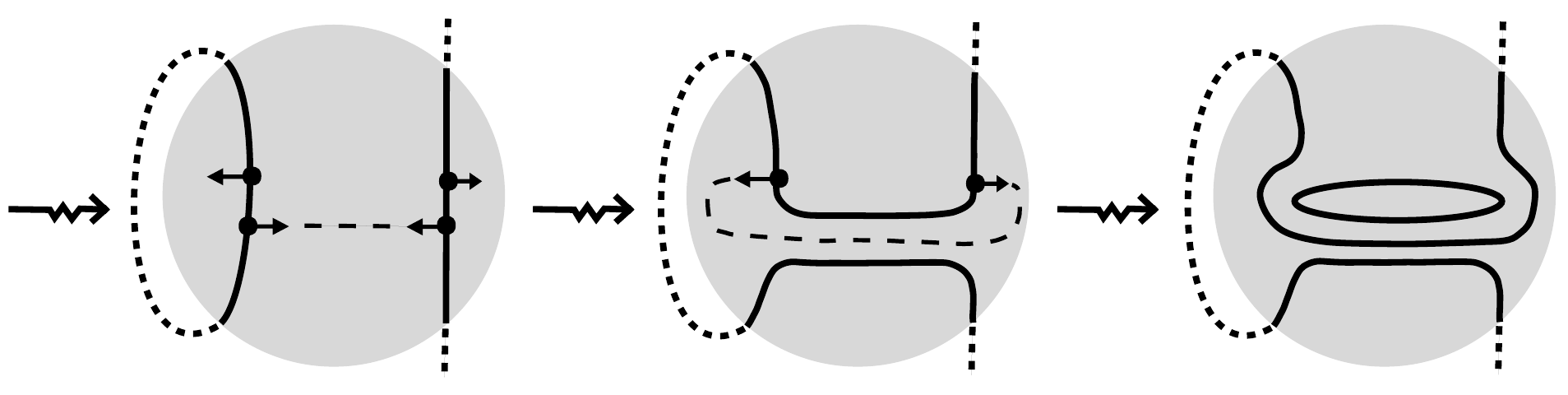}};
    \draw (-5.7, 0.25) node {$2 \times$(C)};
    \draw (-1.6, 0.25) node {(E)};
    \draw (2.5, 0.25) node {(E)};
    \draw (-4.8, 1.3) node {$C$};
\end{tikzpicture}}
\caption{Loop simplification.
The modification is performed as indicated on an embedded open disc in $W \setminus \partial W$ containing $\alpha([0, 1])$.}
\label{figure loop simplification}
\end{figure}
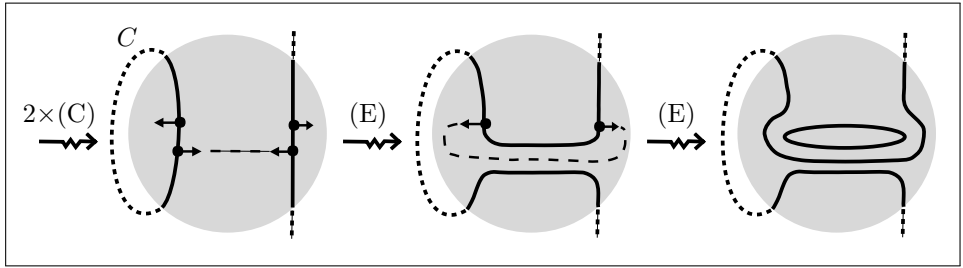

\begin{figure}[htbp]
\centering
\fbox{\begin{tikzpicture}
    \draw (0, 0) node {\includegraphics[height=0.32\textwidth]{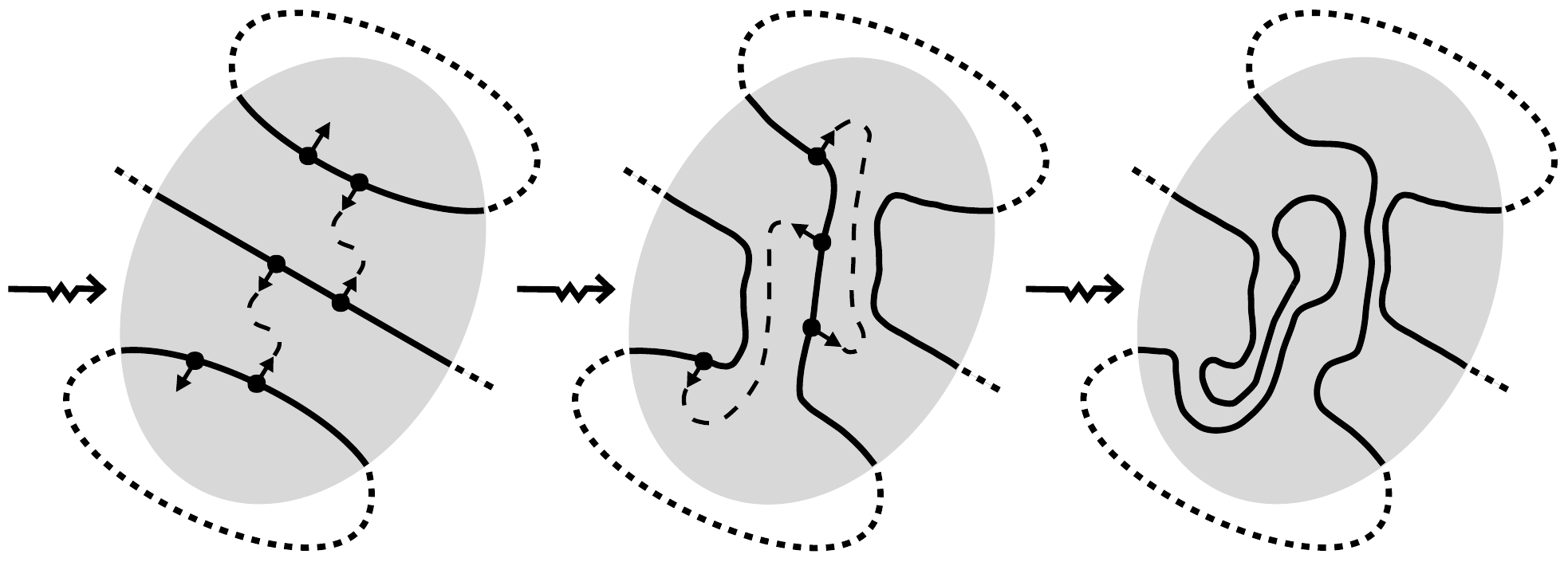}};
    \draw (-5.3, 0.25) node {$3 \times$(C)};
    \draw (-1.6, 0.25) node {$2 \times$(E)};
    \draw (-1.6, -0.39) node {(C)};
    \draw (2.1, 0.25) node {$2 \times$(E)};
    \draw (-2.6, -1.7) node {$C_{0}$};
    \draw (-4.3, 1.7) node {$C_{1}$};
\end{tikzpicture}}
\caption{Loop reduction.
The modification is performed as indicated on an embedded open disc in $W \setminus \partial W$ containing $\alpha([0, 1])$.}
\label{figure loop reduction}
\end{figure}
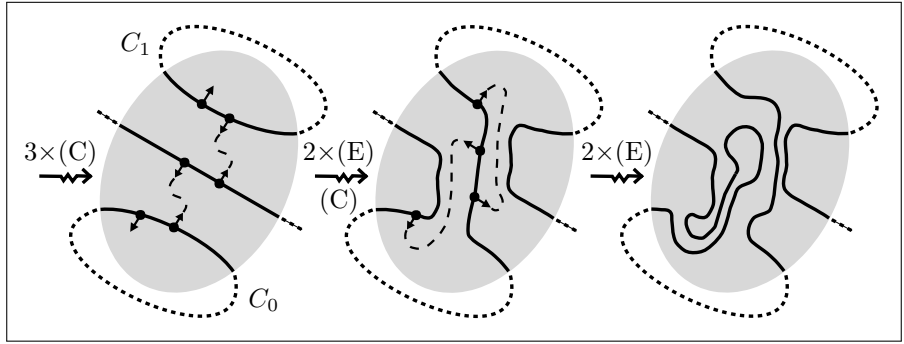

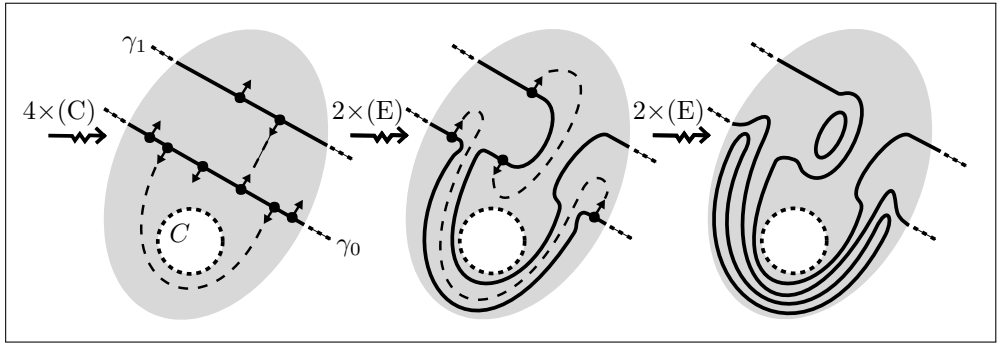
\begin{figure}[htbp]
\centering
\fbox{\begin{tikzpicture}
    \draw (0, 0) node {\includegraphics[height=0.32\textwidth]{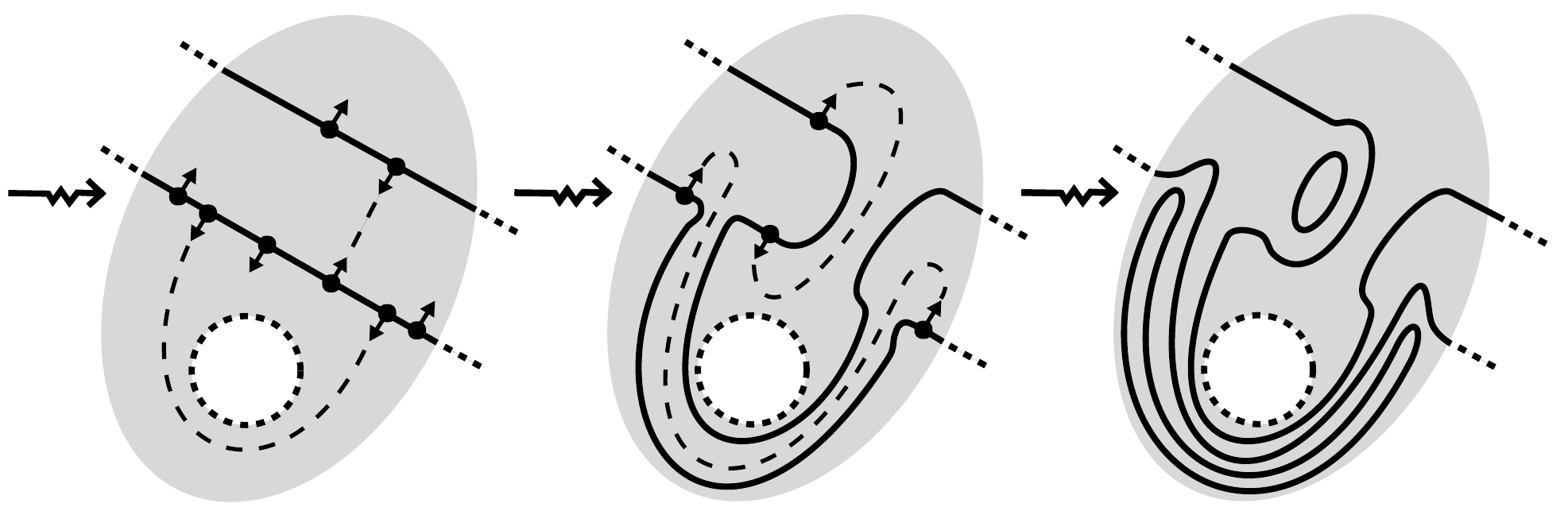}};
    \draw (-6, 0.8) node {$4 \times$(C)};
    \draw (-1.9, 0.8) node {$2 \times$(E)};
    \draw (2.1, 0.8) node {$2 \times$(E)};
    \draw (-4.4, -0.8) node {$C$};
    \draw (-2.15, -1.0) node {$\gamma_{0}$};
    \draw (-5.0, 1.7) node {$\gamma_{1}$};
\end{tikzpicture}}
\caption{Tunneling.
The modification is performed as indicated on an embedded open annulus in $U$ containing $C' \cup \alpha([s, t_{1}])$.
The construction is completed by using loop reduction (\Cref{proposition loop reduction}) to eliminate the two contractible loops in the figure on the right.}
\label{figure tunneling}
\end{figure}

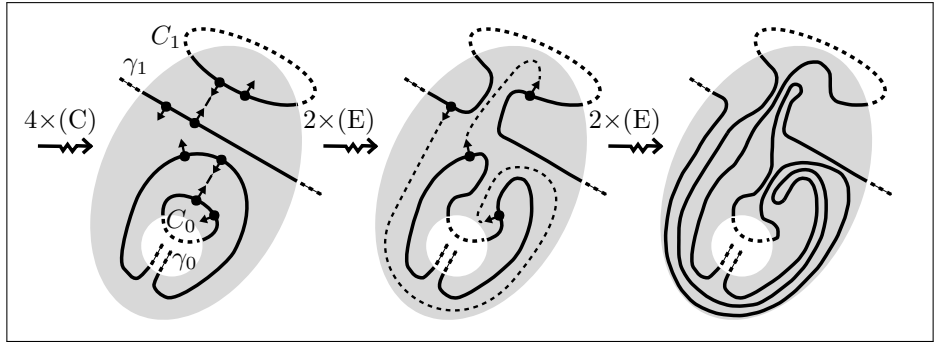
\begin{figure}[htbp]
\centering
\fbox{\begin{tikzpicture}
    \draw (0, 0) node {\includegraphics[height=0.32\textwidth]{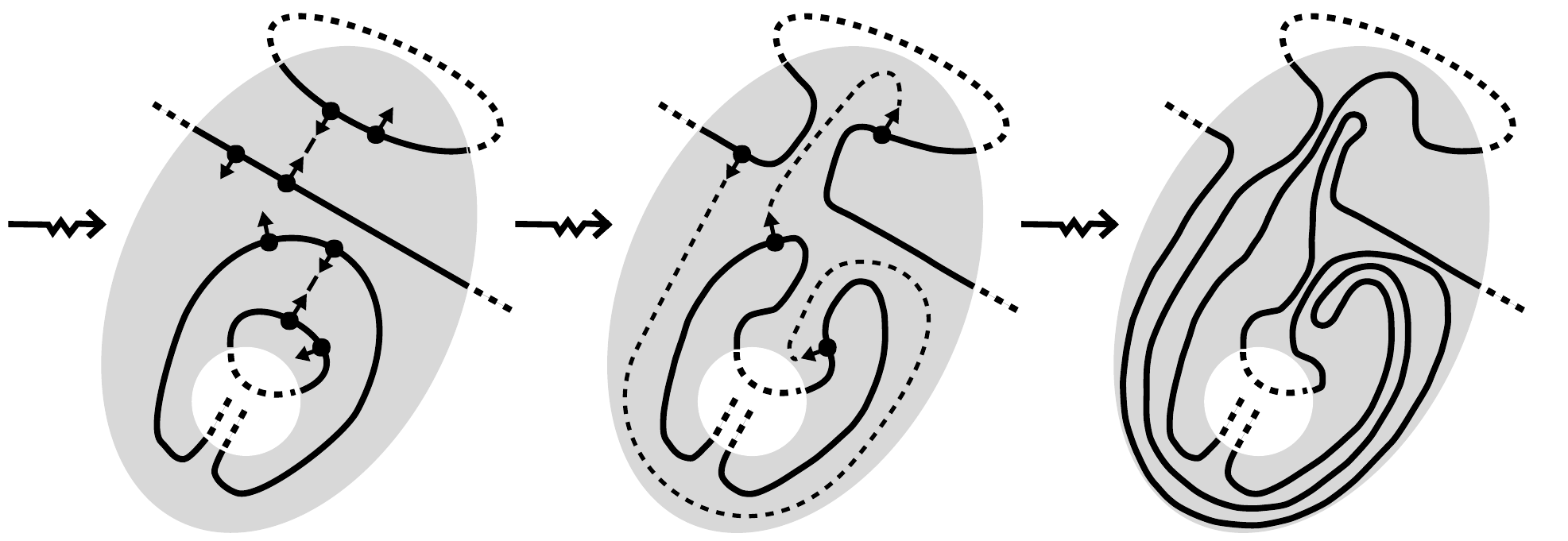}};
    \draw (-5.5, 0.67) node {$4 \times$(C)};
    \draw (-1.8, 0.67) node {$2 \times$(E)};
    \draw (2.0, 0.67) node {$2 \times$(E)};
    \draw (-3.9, -0.65) node {$C_{0}$};
    \draw (-4.1, 1.8) node {$C_{1}$};
    \draw (-3.85, -1.2) node {$\gamma_{0}$};
    \draw (-4.5, 1.4) node {$\gamma_{1}$};
\end{tikzpicture}}
\caption{Balancing.
The modification is performed as indicated on an embedded open annulus in $U$ containing $H \cup \alpha([0, 1])$.}
\label{figure balancing}
\end{figure}

\newpage
\section{Boundary Turning Invariant}\label{Boundary Turning Invariant}

Throughout this section let $P$ denote a closed $1$-dimensional manifold equipped with an orientation $\sigma$.
Furthermore, let $f$ be a boundary condition on $P$.
That is, $f \colon (-\varepsilon, \varepsilon) \times P \rightarrow \mathbb{R}^{2}$, $\varepsilon > 0$, is a fold map whose singular locus $S(f) \subset (-\varepsilon, \varepsilon) \times P$ is transverse to $\{0\} \times P$.
In particular, the set $P_{f} := S(f) \cap \{0\} \times P$ is finite.
To these data we will assign in \Cref{boundary turning invariant} an integer $\omega_{\sigma}(f)$ called the boundary turning invariant.

\begin{lemma}\label{lemma boundary data}
\begin{enumerate}[(a)]
\item There exists $\varepsilon' \in (0, \varepsilon)$ such that the projection to the first factor $(-\varepsilon', \varepsilon') \times P \rightarrow (-\varepsilon', \varepsilon')$ restricts to a local diffeomorphism $S(f) \cap (-\varepsilon', \varepsilon') \times P \rightarrow (-\varepsilon', \varepsilon')$.
\item For every component $C$ of $P$, the (possibly empty) set $C_{f} := S(f) \cap \{0\} \times C$ is finite, and has even cardinality.
\end{enumerate}
\end{lemma}

\begin{definition}\label{definition adapted embeddings}
Let $C$ be a component of $P$.
An embedding $\alpha \colon C \rightarrow (-\varepsilon, \varepsilon) \times C$ is called \emph{$(\sigma, f)$-adapted} if the following properties are satisfied:
\begin{enumerate}[(i)]
\item There exists $\varepsilon' \in (0, \varepsilon)$ as in \Cref{lemma boundary data}(a) such that $\alpha(C) \subset (-\varepsilon', \varepsilon') \times C$.
\item $\alpha$ is transverse to $S(f)$, and $\alpha(C) \cap S(f) = C_{f}$ (see \Cref{lemma boundary data}(b)).
\item The composition of $\alpha$ with the projection to the second factor $(-\varepsilon, \varepsilon) \times C \rightarrow C$ has degree $+1$.
\item The composition $f \circ \alpha \colon C \rightarrow \mathbb{R}^{2}$ is an immersion.
\end{enumerate}
\end{definition}

\Cref{figure adapted embeddings} illustrates the behavior of a $(\sigma, f)$-adapted embedding $\alpha \colon C \rightarrow (-\varepsilon, \varepsilon) \times C$ around a point $x \in C_{f}$.
Let $u_{x} \in T_{x} S(f)$ denote a non-zero tangent vector which points into $(0, \varepsilon) \times C$.
Let $v_{x} \in \operatorname{ker} d_{x}f$ be a non-zero vector such that the pair $(u_{x}, v_{x})$ gives the product orientation of $(-\varepsilon, \varepsilon) \times C$ at $x$.
If $w_{x} \in T_{x}(\alpha(C))$ denotes a non-zero tangent vector that gives the orientation of $\alpha(C)$ induced by the orientation $\sigma|_{C}$, then properties (i), (ii) and (iii) of \Cref{definition adapted embeddings} imply that the pair $(u_{x}, w_{x})$ is a basis of $T_{x}((-\varepsilon, \varepsilon) \times C)$ which gives the product orientation of $(-\varepsilon, \varepsilon) \times C$.
Moreover, property (iv) implies that the pair $(v_{x}, w_{x})$ is a basis of $T_{x}((-\varepsilon, \varepsilon) \times C)$.

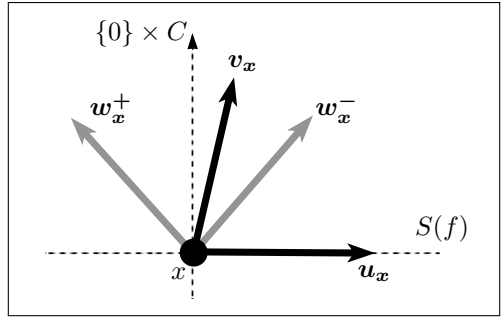
\begin{figure}[htbp]
\centering
\fbox{\begin{tikzpicture}
\draw (0, 0) node {\includegraphics[width=0.48\textwidth]{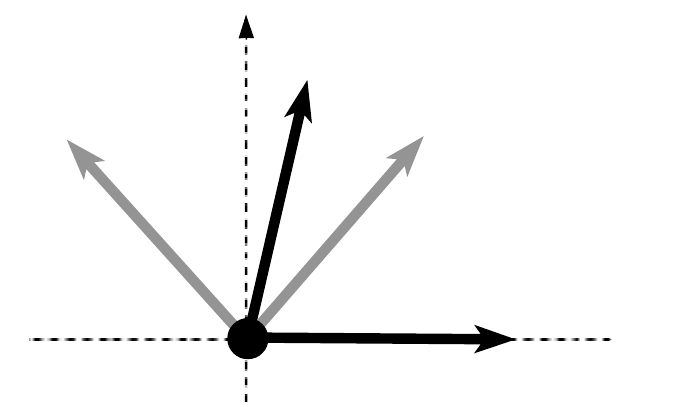}};
\draw (-1.0, -1.5) node {$x$};
\draw (-1.5, 1.7) node {$\{0\} \times C$};
\draw (2.5, -0.9) node {$S(f)$};
\draw (1.6, -1.5) node {$\boldsymbol{u_{x}}$};
\draw (-0.15, 1.3) node {$\boldsymbol{v_{x}}$};
\draw (1.1, 0.7) node {$\boldsymbol{w_{x}^{-}}$};
\draw (-1.9, 0.7) node {$\boldsymbol{w_{x}^{+}}$};
\end{tikzpicture}}
  \caption{A neighborhood in $(-\varepsilon, \varepsilon) \times C$ of a point $x \in C_{f}$.
Given a $(\sigma, f)$-adapted embedding $\alpha$, the vector $w_{x}$ looks like $w_{x}^{+}$ or $w_{x}^{-}$ depending on whether $x$ is positive or negative with respect to $\alpha$ (compare \Cref{definition positive and negative points}).
There are no requirements about the positions of $v_{x}$ and $w_{x}$ relative to $\{0\} \times C$.}
  \label{figure adapted embeddings}
\end{figure}

\begin{definition}\label{definition positive and negative points}
A point $x \in C_{f}$ is called \emph{positive} or \emph{negative} with respect to a $(\sigma, f)$-adapted embedding $\alpha \colon C \rightarrow (-\varepsilon, \varepsilon) \times C$ (compare \Cref{figure adapted embeddings}) depending on whether the pair $(v_{x}, w_{x})$ gives the product orientation of $(-\varepsilon, \varepsilon) \times C$ or its opposite orientation, respectively.
Let $C_{f}^{\alpha}$ denote the sets of points of $C_{f}$ which are negative with respect to $\alpha$.
A $(\sigma, f)$-adapted embedding $\alpha \colon C \rightarrow (-\varepsilon, \varepsilon) \times C$ is called \emph{positive} if $C_{f}^{\alpha} = \emptyset$.
\end{definition}

\begin{lemma}\label{lemma winding number}
Let $C$ be a component of $P$.
There exists a positive $(\sigma, f)$-adapted embedding $\alpha \colon C \rightarrow (-\varepsilon, \varepsilon) \times C$.
Moreover, any two such embeddings are smoothly isotopic in the space of all positive $(\sigma, f)$-adapted embeddings.
\end{lemma}

\Cref{lemma winding number} enables us to assign to every component $C$ of $P$ a well-defined integer $\omega_{\sigma|_{C}}(f|_{(-\varepsilon, \varepsilon) \times C})$ given by the winding number $W(f \circ \alpha)$ (i.e., the degree of the Gauss map, compare \cite{whi2}) of the immersion $f \circ \alpha \colon C \rightarrow \mathbb{R}^{2}$ for any positive $(\sigma, f)$-adapted embedding $\alpha \colon C \rightarrow (-\varepsilon, \varepsilon) \times C$.

\begin{definition}\label{boundary turning invariant}
The \emph{boundary turning invariant $\omega_{\sigma}(f)$ of $f$} is defined as
\begin{displaymath}
\omega_{\sigma}(f) := \sum_{C} \omega_{\sigma|_{C}}(f|_{(-\varepsilon, \varepsilon) \times C}) \qquad \in \mathbb{Z},
\end{displaymath}
where the sum runs through the components $C$ of $P$.
\end{definition}

\begin{example}\label{example boundary turning invariant of cusps and fold points}
The embedding $g \colon \mathbb{R} \times S^{1} \rightarrow \mathbb{R}^{2}$, $g(r, s) = e^{r} \cdot s$, defines a tubular neighborhood of the circle $S^{1} \subset \mathbb{R}^{2}$.
Let $f_{\operatorname{fold}}$ denote the composition of $g$ with the fold map $(x, y) \mapsto (x, y^{2})$, and let $f_{\operatorname{cusp}}$ denote the composition of $g$ with the stable Whitney cusp $(x, y) \mapsto (x, xy + y^{3})$.
Then, for any orientation $\sigma$ on the circle $S^{1}$, one can show that $\omega_{\sigma}(f_{\operatorname{fold}}) = 0$ and $\omega_{\sigma}(f_{\operatorname{cusp}}) = \pm 1$.
\end{example}

\begin{lemma}\label{proposition embeddings of boundary turning invariant}
Let $C$ be a component of $P$.
If $\alpha \colon C \rightarrow (-\varepsilon, \varepsilon) \times C$ is a $(\sigma, f)$-adapted embedding, then the winding number of the immersion $f \circ \alpha \colon C \rightarrow \mathbb{R}^{2}$ has the same parity as $\omega_{\sigma|_{C}}(f|_{(-\varepsilon, \varepsilon) \times C}) + |C_{f}^{\alpha}|$.
\end{lemma}

\begin{proof}
Given $x \in C_{f}^{\alpha}$, we show how to modify $\alpha$ in a small neighborhood of $\alpha^{-1}(x)$ to obtain a $(\sigma, f)$-adapted embedding $\beta \colon C \rightarrow (-\varepsilon, \varepsilon) \times C$ such that $x \notin C_{f}^{\beta}$, and such that the winding numbers of the immersions $f \circ \alpha$ and $f \circ \beta$ differ by $\pm 1$.

Let $\mathfrak{X}$ denote the space (equipped with the Whitney $C^{\infty}$ topology) of all smooth maps $\gamma \colon C \rightarrow (-\varepsilon, \varepsilon) \times C$ for which the composition $f \circ \gamma$ is an immersion.
Note that the winding number of $f \circ \gamma$ is invariant under homotopies of $\gamma$ in $\mathfrak{X}$.

The given $(\sigma, f)$-adapted embedding $\alpha$ is homotopic in $\mathfrak{X}$ to the embedding $\alpha_{1} \colon C \rightarrow (-\varepsilon, \varepsilon) \times C$ indicated in \Cref{figure deformation}(a) by means of a homotopy supported in a small neighborhood of $x$.
Then, the winding numbers of the immersions $f \circ \alpha$ and $f \circ \alpha_{1}$ agree.
Let $\alpha_{2} \colon C \rightarrow (-\varepsilon, \varepsilon) \times C$ denote an immersion which arises from $\alpha_{1}$ by inserting two loops based at $x$ as indicated in \Cref{figure deformation}(b).
Then, the winding numbers of the immersions $f \circ \alpha_{1}$ and $f \circ \alpha_{2}$ differ by $\pm 1$.
Finally, $\alpha_{2}$ is homotopic in $\mathfrak{X}$ to the desired embedding $\beta \colon C \rightarrow (-\varepsilon, \varepsilon) \times C$ (see \Cref{figure deformation}(c)).
\end{proof}

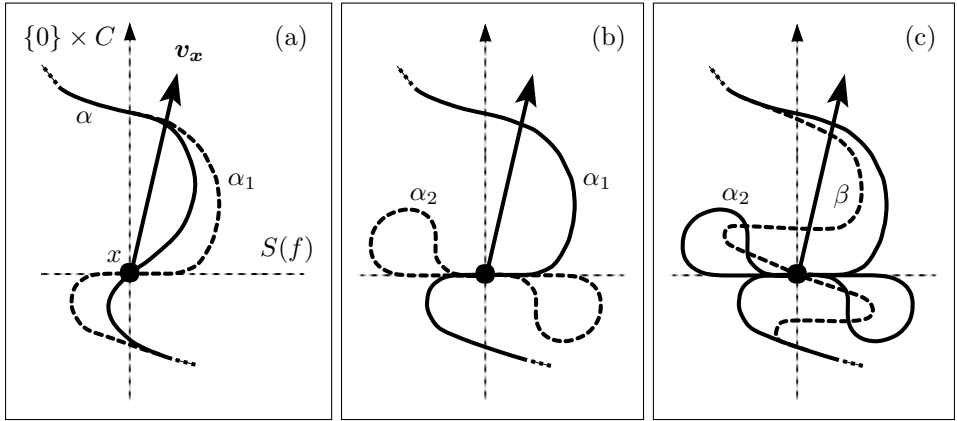
\begin{figure}[htbp]
\centering
\fbox{\begin{tikzpicture}
\draw (0, 0) node {\includegraphics[width=0.28\textwidth]{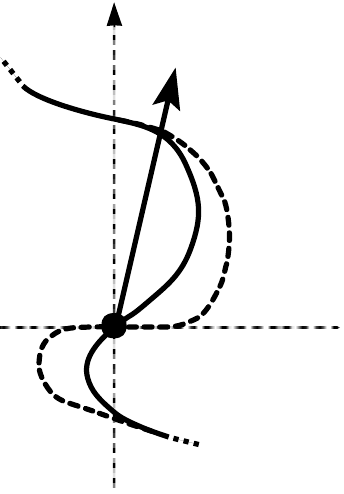}};
\draw (-0.8, -0.6) node {$x$};
\draw (0.2, 2.1) node {$\boldsymbol{v_{x}}$};
\draw (-1.4, 2.3) node {$\{0\} \times C$};
\draw (1.5, -0.5) node {$S(f)$};
\draw (0.9, 0.4) node {$\alpha_{1}$};
\draw (-1.2, 1.25) node {$\alpha$};
\draw (1.55, 2.3) node {(a)};
\end{tikzpicture}}
\fbox{\begin{tikzpicture}
\draw (0, 0) node {\includegraphics[width=0.28\textwidth]{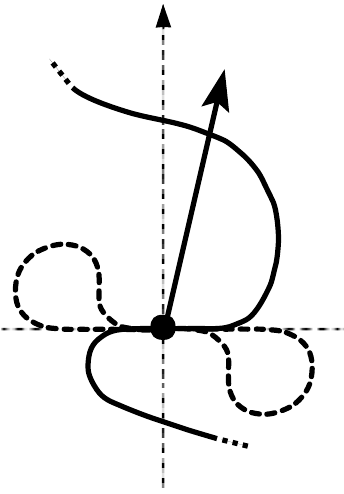}};
\draw (1.4, 0.4) node {$\alpha_{1}$};
\draw (-0.9, 0.2) node {$\alpha_{2}$};
\draw (1.55, 2.3) node {(b)};
\end{tikzpicture}}
\fbox{\begin{tikzpicture}
\draw (0, 0) node {\includegraphics[width=0.28\textwidth]{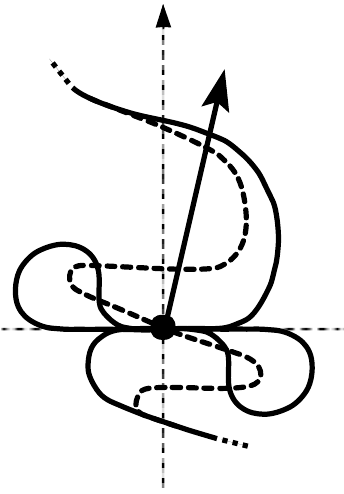}};
\draw (-0.9, 0.2) node {$\alpha_{2}$};
\draw (0.5, 0.2) node {$\beta$};
\draw (1.55, 2.3) node {(c)};
\end{tikzpicture}}
  \caption{Modification of $\alpha$ in a small neighborhood of $\alpha^{-1}(x)$.}
  \label{figure deformation}
\end{figure}

\begin{proposition}\label{corollary independence of parity of boundary turning invariant of orientation}
If $\rho$ is any orientation of $P$, then $\omega_{\sigma}(f) \equiv \omega_{\rho}(f) \, \operatorname{mod} 2$.
\end{proposition}

\begin{proof}
It suffices to show that $\omega_{\sigma|_{C}}(f|_{(-\varepsilon, \varepsilon) \times C}) \equiv \omega_{-\sigma|_{C}}(f|_{(-\varepsilon, \varepsilon) \times C}) \, \operatorname{mod} 2$ for every component $C$ of $P$.
For this purpose, note that if $\alpha \colon C \rightarrow (-\varepsilon, \varepsilon) \times C$ is a positive $(\sigma, f)$-adapted embedding, then $\beta := \alpha \circ \iota$ is a $(-\sigma, f)$-adapted embedding such that $C_{f}^{\beta} = C_{f}$, where $\iota \colon C \rightarrow C$ denotes an orientation reversing automorphism.
Hence, the claim follows from \Cref{proposition embeddings of boundary turning invariant} because $C_{f}$ has even cardinality by \Cref{lemma boundary data}(b).
\end{proof}

\newpage
\section{Number of Cusps of Generic Extensions}\label{Number of Cusps of Generic Extensions}

Let $f$ be a boundary condition on $P = \partial W$.
By an \emph{extension of $f$ to $W$} we mean in the following a generic map $F \colon W \rightarrow \mathbb{R}^{2}$ such that $F|_{[0, \varepsilon') \times \partial W} =  f|_{[0, \varepsilon') \times \partial W}$ for some $\varepsilon' \in (0, \varepsilon)$.
Note that any $f$ admits an extension, and any two extensions have the same parity of number of cusps.
The purpose of this section is to prepare the ingredients needed to derive a formula for this parity in terms of the boundary turning invariant (see \Cref{proposition parity of number of cusps of realizations}).

\begin{definition}\label{definition orientation preserving part of W}
Suppose that $W$ is oriented, and let $\sigma$ be the induced orientation on the boundary $P = \partial W$.
For any extension $F \colon W \rightarrow \mathbb{R}^{2}$ of $f$ let $W_{F}^{\sigma}$ denote the closure in $W$ of the union of components $U$ of $W \setminus S(F)$ for which the local diffeomorphism $F|_{U} \colon U \rightarrow \mathbb{R}^{2}$ is orientation preserving.
\end{definition}

Under the assumptions of \Cref{definition orientation preserving part of W}, note that $W_{F}^{\sigma}$ is a compact $2$-dimensional submanifold of $W$ with corners, and the set of corners is $P_{f} = S(f) \cap P$.
Moreover, we have $W = W_{F}^{\sigma} \cup W_{F}^{-\sigma}$ and $\partial W_{F}^{\sigma} \cap \partial W_{F}^{-\sigma} = S(F) = W_{F}^{\sigma} \cap W_{F}^{-\sigma}$.
Note that every component of $S(F)$ is contained in the boundary of exactly two components of $W \setminus S(F)$.
One of these components belongs to $W_{F}^{\sigma}$, and the other component belongs to $W_{F}^{-\sigma}$.

The following proposition relates the boundary turning invariant as introduced in \Cref{Boundary Turning Invariant} to the Euler characteristic.

\begin{proposition}\label{proposition boundary turning invariant and Euler characteristic}
Suppose that $W$ is oriented, and let $\sigma$ be the induced orientation on $P = \partial W$.
Then, for any extension $F \colon W \rightarrow \mathbb{R}^{2}$ of $f$ without cusps, we have
$$
\chi(W_{F}^{\sigma}) - \chi(W_{F}^{-\sigma}) = \omega_{\sigma}(f).
$$
Consequently, it follows from $\chi(W_{F}^{\sigma}) + \chi(W_{F}^{-\sigma}) = \chi(W) + |P_{f}|/2$ that
$$
\chi(W_{F}^{\sigma}) = (\chi(W) + |P_{f}|/2 +\omega_{\sigma}(f))/2.
$$
\end{proposition}

\begin{proof}
Let $V$ be a component of $W_{F}^{\sigma}$.
If $V \cap \partial W = \emptyset$, then $V \subset \operatorname{int} W$ is a submanifold with boundary $V \cap S(F)$, and we set $\widetilde{V} = V$.
If, however, $V \cap \partial W \neq \emptyset$, then $V \subset W$ is a submanifold with corners, and the set of corners is given by $V \cap P_{f}$.
In this case, we define a smooth submanifold with boundary $\widetilde{V} \subset V$ by smoothly cutting off the corners of $V$ in sufficiently small neighborhoods of the corner points as indicated in \Cref{figure smoothly cutting corners}.

\begin{figure}[htbp]
  \centering
\fbox{\begin{tikzpicture}
\draw (0, 0) node {\includegraphics[width=0.6\textwidth]{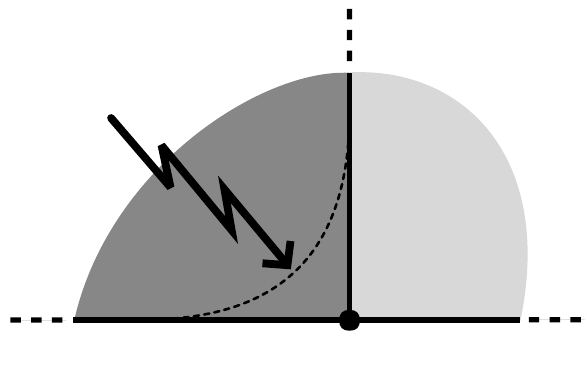}};
\draw (0.7, -2.1) node {$x$};
\draw (-2.6, 1.3) node {$\partial \widetilde{V}$};
\draw (-3.1, -2.0) node {$P = \partial W$};
\draw (0.2, 2.1) node {$S(f)$};
\end{tikzpicture}}
\caption{A neighborhood (grey) in $[0, \varepsilon') \times P$ of a corner point $x \in V \cap P_{f}$ of a component $V$ of $W_{F}^{\sigma}$ (dark grey) for sufficiently small $\varepsilon' \in (0, \varepsilon)$.
The boundary of $\widetilde{V}$ deviates from that of $V$ as indicated only in small neighborhoods of the corners of $V$.}
\label{figure smoothly cutting corners}
\end{figure}
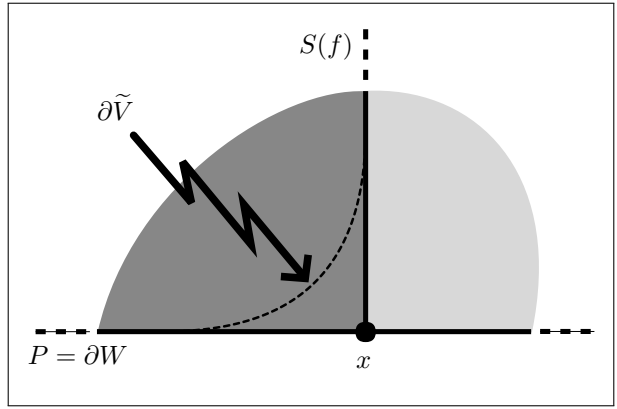

In any case, a theorem of Haefliger \cite{hae} implies that the winding number of the immersion $F|_{\partial \widetilde{V}} \colon \partial \widetilde{V} \rightarrow \mathbb{R}^{2}$ equals $\chi(V)$, where $\partial \widetilde{V}$ is equipped with the orientation induced by that of $\widetilde{V} \subset W$.
Analogously, while keeping the same orientation on $W$, we can identify the Euler characteristic $\chi(V)$ of any component $V$ of $W_{F}^{-\sigma}$ with the negative of the winding number of the immersion $F|_{\partial \widetilde{V}} \colon \partial \widetilde{V} \rightarrow \mathbb{R}^{2}$, where $\widetilde{V} \subset V$ is a suitably chosen smooth submanifold with boundary.
Hence,
$$
\chi(W_{F}^{\sigma}) - \chi(W_{F}^{-\sigma}) = \sum_{V \subset W_{F}^{\sigma} \sqcup W_{F}^{-\sigma}} \text{winding number of } \partial \widetilde{V},
$$
where the sum runs through all components of $W_{F}^{\sigma} \sqcup W_{F}^{-\sigma}$.
We observe that the contributions to the sum of winding numbers by adjacent components will cancel each other along common parts of their boundaries.
Thus, only contributions of boundary parts contained in a neighborhood of $[0, \varepsilon') \times P$ for sufficiently small $\varepsilon' \in (0, \varepsilon)$ remain (compare \Cref{figure smoothly cutting corners}).
(Our argument requires an extension of the concept of winding number to immersions of arcs into $\mathbb{R}^{2}$.
This generalized winding number for immersions arbitrary compact $1$-manifolds turns out to behave additively under disjoint union and gluing of $1$-manifolds along boundaries.
By definition, the winding number of an immersion $\gamma \colon [a, b] \rightarrow \mathbb{R}^{2}$ is the real number $\lambda(b) - \lambda(a)$, where $\lambda \colon [a, b] \rightarrow \mathbb{R}$ denotes a lift of the Gauss map $[a, b] \rightarrow S^{1}$ of $\gamma$ with respect to the universal cover $\mathbb{R} \rightarrow S^{1}$, $t \mapsto e^{2 \pi t}$.
We refer to Chapter 1 of \cite{wra2} for details in the context of positive TFTs.)
Finally, an argument in the spirit of the proof of \Cref{proposition embeddings of boundary turning invariant} identifies the resulting sum of (generalized) winding numbers with $\omega_{\sigma}(f)$.
\end{proof}

\section{Patterns and their Realizations}

\begin{definition}\label{pattern}
A \emph{(singular) pattern on $W$} is a pattern $(f, \varphi)$ consisting of
\begin{itemize}
\item a fold map $f \colon (-\varepsilon, \varepsilon) \times P \rightarrow \mathbb{R}^{2}$ for some $\varepsilon > 0$ whose singular locus $S(f) \subset (-\varepsilon, \varepsilon) \times P$
(which is well-known to be a $1$-dimensional submanifold) is transverse to $\{0\} \times P$, and
\item a fixed-point free involution $\varphi \colon P_{f} \rightarrow P_{f}$ of the finite set $P_{f} := S(f) \cap (\{0\} \times P)$.
(Note that $\varphi$ can equivalently be considered as a partition of $P_{f}$ into subsets of cardinality $2$.)
\end{itemize}
\end{definition}

\begin{definition}\label{definition realizations}
A $1$-dimensional submanifold $S \subset W$ satisfying $S \pitchfork \partial W$ and $\partial S = S \cap \partial W$ is said to be \emph{adapted to a pattern $(f, \varphi)$} if for every $x \in P_{f}$ there exists a component $S_{x}$ of $S$ with boundary $\partial S_{x} = \{x, \varphi(x)\}$.
By a \emph{realization} of the pattern $(f, \varphi)$ we mean a generic extension $F \colon W \rightarrow \mathbb{R}^{2}$ of $f|_{[0, \varepsilon') \times \partial W}$ for some $\varepsilon' \in (0, \varepsilon)$ such that the singular locus $S(F)$ is non-empty and adapted to $\varphi$.
\end{definition}

\begin{remark}\label{remark singular locus not empty}
Requiring that $S(F) \neq \emptyset$ for realizations $F$ of the pattern $(f, \varphi)$ excludes the case that $F$ is an immersion.
The extension problem for stable immersions $\partial W \rightarrow \mathbb{R}^{2}$ to immersions $W \rightarrow \mathbb{R}^{2}$ has been chracterized combinatorially in \cite{fra2}.
\end{remark}

Fix a pattern $(f, \varphi)$.
A map $\iota \colon P_{f} \rightarrow \{-1, +1\}$ is defined by assigning to every point $x \in P_{f}$ a positive or a negative sign according to whether a non-zero tangent vector $u_{x} \in T_{x} S(f)$ which points into $(0, \varepsilon) \times \partial W$ determines the canonical orientation of $S(F)$ at $x$ (see \Cref{definition canonical orientation}) or its opposite orientation.
If $F$ is a realization of $(f, \varphi)$, then $S(F)$ is adapted to $\varphi$, and $\iota(x) \neq \iota(\varphi(x))$ for all $x \in P_{f}$.

\begin{proposition}\label{proposition existence of realizations}
The pattern $(f, \varphi)$ admits a realization if and only if
\begin{enumerate}[(i)]
\item $\iota(x) \neq \iota(\varphi(x))$ for all $x \in P_{f}$, and
\item there exists a submanifold $S \subset W$ which is adapted to $(f, \varphi)$.
\end{enumerate}
\end{proposition}

\begin{proof}
It suffices to show that a realization $F$ of $(f, \varphi)$ exists under the given conditions.
First, we produce a realization $F'$ of some pattern $(f, \varphi')$ by extending the germ of $f$ at $\partial W$ generically over $W$ in an arbitrary way.
Then, for every point $x \in P_{f}$, we use (C) once to create a new pair of cusps sufficiently close to $x$ on $S(F) \setminus \partial W$.
Next, for every pair $\{x, y\} \subset P_{f}$ of the partition $\varphi$ we eliminate the cusp of $F$ which is closest to $x$ and the cusp of $F$ which is closest to $y$ by applying general cusp elimination (\Cref{lemma generalized cusp elimination}) along a path in $W \setminus \partial W$ that connects these two cusps, and whose interior intersects $S(F)$ transversely and only in fold points of $F$.
As a result, in view of assumption (i), the singular locus of the modified generic map $\widetilde{F}$ will contain a component with boundary $\{x, y\}$, and which contains no fold points.
Finally, note that assumption (ii) makes sure that we can perform all these cusp eliminations independently of each other along pairwise disjoint paths in $W \setminus \partial W$.

\end{proof}

\section{Proof of \Cref{MAIN THEOREM 1}}\label{proof of main theorem 1}

Part (a) follows immediately from the following

\begin{proposition}\label{proposition parity of number of cusps of realizations}
Given an orientation $\sigma$ of $P = \partial W$, the number of cusps of any realization $F \colon W \rightarrow \mathbb{R}^{2}$ of a given pattern $(f, \varphi)$ on $W$ has the same parity as the quantity
\begin{displaymath}
\Gamma^{\sigma} := \chi(W) + |P_{f}|/2 + \omega_{\sigma}(f).
\end{displaymath}
\end{proposition}

\begin{proof}
Let $F \colon W \rightarrow \mathbb{R}^{2}$ be an extension of $f$.
First we show in part $(1)$ of the proof that $\Gamma^{\sigma}$ is even when $F$ has an even number of cusps.
This will then be used in part $(2)$ of the proof to show that $\Gamma^{\sigma}$ is odd when $F$ has an odd number of cusps.

$(1)$.
Suppose that $F$ has an even number of cusps.

First suppose that $W$ is orientable.
Then, we may assume by \Cref{corollary independence of parity of boundary turning invariant of orientation} that $W$ is oriented in such a way that $\sigma$ is the induced orientation of $P$.
As $F$ has an even number of cusps, we may assume by \Cref{lemma generalized cusp elimination} that $F$ has no cusps.
Then, the claim that $\Gamma^{\sigma}$ is even follows from \Cref{proposition boundary turning invariant and Euler characteristic}.

Now suppose that $W$ is non-orientable.
Let $C \subset \operatorname{int} W$ be an embedded circle whose normal line bundle $\nu \colon N \rightarrow C$ is non-trivial, where $N \subset W$ denotes a tubular neighborhood of $C$ in $W$.
Without loss of generality we may assume that $C$ is transverse to $S(F)$, and avoids the cusps of $F$.
If we cut $W$ along $C$, then the resulting manifold $W_{0}$ has one additional boundary component $C_{0}$.
By construction there is a quotient map $\pi \colon W_{0} \rightarrow W$ that restricts to a diffeomorphism $W_{0} \setminus C_{0} \stackrel{\cong}{\longrightarrow} W \setminus C$, and to a covering map $\gamma \colon C_{0} \rightarrow C$ of degree $2$.
The pullback bundle of $\nu \colon N \rightarrow C$ under $\gamma \colon C_{0} \rightarrow C$ is the trivial line bundle over $C_{0}$, and $\gamma$ can be covered by a canonical bundle map $\widetilde{\gamma} \colon \mathbb{R} \times C_{0} \rightarrow N$ such that $\widetilde{\gamma}(r, -x) = -\widetilde{\gamma}(r, x)$, where $-x$ denotes the unique element in the fiber $\gamma^{-1}(\gamma(x))$ that is distinct from $x \in C_{0}$.
Moreover, there exists a collar $[0, \infty) \times C_{0} \subset W_{0}$ of $C_{0}$ in $W_{0}$ such that $\pi \colon W_{0} \rightarrow W$ and $\widetilde{\gamma} \colon \mathbb{R} \times C_{0} \rightarrow N$ restrict to the same map $[0, \infty) \times C_{0} \rightarrow N$.
Let $\rho$ denote an orientation of $C_{0}$.
Then, writing $g := F \circ \widetilde{\gamma} \colon \mathbb{R} \times C_{0} \rightarrow \mathbb{R}^{2}$, it follows from \Cref{proposition embeddings of boundary turning invariant} that $\omega_{\rho}(g)$ has the same parity as the cardinality of $C \cap S(F)$.
(In fact, this can be shown by choosing an embedding $\alpha \colon C \rightarrow N$ for which the lift of $\alpha \circ \gamma$ to a map $\tilde{\alpha} \colon C_{0} \rightarrow \mathbb{R} \times C_{0}$ is a $(\rho, g)$-adapted embedding in the sense of \Cref{definition adapted embeddings}.)
The quotient map $\pi \colon W_{0} \rightarrow W$ restricts to a diffeomorphism $\pi_{0} \colon P_{0} = \pi^{-1}(P) \stackrel{\cong}{\longrightarrow} P$.
Let $\sigma_{0}$ denote the orientation of $P_{0}$ induced by $\sigma$ via $\pi_{0}$.
The collar $[0, \infty) \times P \subset W$ induces a collar $[0, \infty) \times P_{0} \subset W_{0}$ by means of $(t, x) \mapsto \pi^{-1}(t, \pi_{0}(x))$.
Let $f_{0} = f \circ \pi|_{[0, \infty) \times P_{0}}$.
Then,
$$\omega_{\sigma_{0} \sqcup \rho}(f_{0} \sqcup g) = \omega_{\sigma_{0}}(f_{0}) + \omega_{\rho}(g) \equiv \omega_{\sigma}(f) + |C \cap S(F)| \, \operatorname{mod} 2.$$
From now on, we assume that $C$ is chosen in such a way that $W_{0}$ is orientable.
Then, since $F_{0} := F \circ \pi$ is an extension of $f_{0} \sqcup g$ to $W_{0}$ with an even number of cusps, we have already shown above that
$$\chi(W_{0}) + |(P_{0} \sqcup C_{0})_{f_{0} \sqcup g}|/2 + \omega_{\sigma_{0} \sqcup \rho}(f_{0} \sqcup g) \equiv 0 \, \operatorname{mod} 2.$$
By construction, $\chi(W_{0}) = \chi(W)$.
Moreover, note that we have
$$|(P_{0} \sqcup C_{0})_{f_{0} \sqcup g}| = |P_{f}| + 2 |C \cap S(F)|.$$
All in all, this completes the proof of part $(1)$.

$(2)$.
Suppose that $F$ has an odd number of cusps, and let $c$ denote a cusp of $F$.
Let $U \subset \operatorname{int} W$ be a small open disc centered at $c$ in which $F$ looks like the stable Whitney cusp discussed in \Cref{example boundary turning invariant of cusps and fold points}.
Recall that $\omega_{\rho}(f_{\operatorname{cusp}}) = \pm 1$, where $\rho$ denotes an orientation of $\partial U \cong S^{1}$.
For $W_{0} = W \setminus U$ and the restriction $F_{0} = F|_{W_{0}}$ we know by part $(1)$ that
$$
\chi(W_{0}) + |(\partial W_{0})_{f \sqcup f_{\operatorname{cusp}}}|/2 + \omega_{\sigma \sqcup \rho}(f \sqcup f_{\operatorname{cusp}}) \equiv 0 \, \operatorname{mod} 2.
$$
By construction, $\chi(W) = \chi(W_{0}) + 1$, $|(\partial W_{0})_{f \sqcup f_{\operatorname{cusp}}}| = |P_{f}| + 2$, and $\omega_{\sigma \sqcup \rho}(f \sqcup f_{\operatorname{cusp}}) = \omega_{\sigma}(f) \pm 1$.
Thus, the claim follows.
\end{proof}

\begin{remark}\label{remark Fukuda-Ishikawa theorem}
Compare \Cref{proposition parity of number of cusps of realizations} to the Fukuda-Ishikawa theorem (see Corollary 1.2 in \cite[p. 377]{fukishi}), where the target manifold is a compact surface $N$, and the extension $F \colon W \rightarrow N$ of boundary conditions restricts to a Morse function $\partial W \rightarrow \partial N$.
In this context, in place of the boundary turning invariant, one has to use the related notion of degree of a continuous map $S^{1} \rightarrow S^{1}$.
\end{remark}

As for the proof of part (b) of \Cref{MAIN THEOREM 1}, let $F \colon W \rightarrow \mathbb{R}^{2}$ be a realization of the pattern $(f, \varphi)$.
An inspection of the proof of \Cref{proposition existence of realizations} shows that $F$ can be modified on $W \setminus \partial W$ by a finite sequence of (E) and (C) moves to obtain a realization $F_{1}$ of $(f, \varphi)$ whose cusps lie all on loops of $F_{1}$.
Let $T \subset W$ denote the union of the components of $S(F_{1})$ that are diffeomorphic to the unit interval.
Using general cusp elimination (\Cref{lemma generalized cusp elimination}), we can modify $F_{1}$ on $W \setminus (T \cup \partial W)$ by a finite sequence of (E) and (C) moves to obtain a realization $F_{2}$ of $(f, \varphi)$ such that each component of $W \setminus T$ contains at most one cusp of $F_{2}$.
If $T = \emptyset$, then $F_{2}$ has by construction at most one cusp as desired.
Otherwise, fix a component $J$ of $T$.
Note that every cusp $x$ of $F_{2}$ is the unique cusp on some loop of $F_{2}$, say $C_{x}$.
We may also assume that $C_{x}$ is trivial (see \Cref{definition loops}).
(In fact, if the normal line bundle of $C_{x} \subset W$ is non-trivial, then we apply (C) once at a fold point of $F_{2}$ on $C_{x}$ to create a new pair of cusps, and then we apply (E) to a joining curve between these two cusps that goes around once in a small tubular neighborhood of $C_{x}$.
The resulting new loop containing the cusp will then be trivial.)
By a finite iteration of the loop tunneling move (\Cref{proposition tunneling}) we can modify $F_{2}$ on $W \setminus (J \cup \partial W)$ by a finite sequence of (E) and (C) moves to obtain a realization $F_{3}$ of $(f, \varphi)$ such that every cusp $c$ of $F_{3}$ is the unique cusp on some loop of $F_{3}$, and there exists a path $\alpha_{c} \colon [0, 1] \rightarrow W \setminus \partial W$ such that $\alpha_{c}^{-1}(S(F_{3})) = \{0, 1\}$ and $\alpha_{c}(0) = c$, $\alpha_{c}(1) \in J$.
Furthermore, we may assume that for every cusp $c$ of $F_{3}$ the vector $\alpha_{c}'(0)$ is pointing downward at $c$ (see \Cref{local moves}).
(In fact, this can always be achieved by one (C) move followed by one (E) move near $c$.)
Two cusps $c$ and $d$ of $F_{3}$ for which $\alpha_{c}'(1)$ and $\alpha_{d}'(1)$ point to different sides of $J$ can be eliminated by first applying (C) to generate a new cusp pair on $J$, and then applying (E) twice along joining curves given by suitable perturbations of $\alpha_{c}$ and $\alpha_{d}$.
Repeating this process until no such cusps pairs are left, the remaining cusps lie all in the same component of $W \setminus T$, and can thus be eliminated by means of general cusp elimination (\Cref{lemma generalized cusp elimination}) up to at most one cusp.

This completes the proof of \Cref{MAIN THEOREM 1}.

\section{Realizations and Number of Loops}\label{Realizations and Number of Loops}

Throughout the present section, we suppose that $W$ is orientable, and that $\sigma$ denotes an orientation of $P = \partial W$ which is induced by one of the two orientations of $W$.
The statement of our main result \Cref{MAIN THEOREM 2}(b) involves the following quantities derived from a pattern $(f, \varphi)$ on $W$ and the fixed orientation $\sigma$ of $\partial W$:
\begin{itemize}
\item $n_{\sigma}(f)$ denotes the number of components $C$ of $\partial W$ with the property that $f$ restricts to an orientation preserving immersion $(-\varepsilon', \varepsilon') \times C \rightarrow \mathbb{R}^{2}$ for some $\varepsilon' \in (0, \varepsilon)$.
(Here, we assume that $(-\varepsilon', \varepsilon') \times C$ is oriented in such a way that the inclusion $[0, \varepsilon') \times C \subset W$ is orientation preserving.
Thus, $(-\varepsilon', \varepsilon') \times C$ is equipped with the orientation opposite to the product orientation.)
\item $c_{\sigma}(f, \varphi)$ denotes the number of components of the closed $1$-manifold $Z_{\sigma}(f, \varphi)$ obtained as follows.
Let $\pi_{\sigma}(f)$ be the fixed-point free involution of $P_{f} = (\{0\} \times P) \cap S(f)$ that corresponds to the partition of $P_{f}$ into pairs of the form $P_{f} \cap \overline{V}$, where $V$ runs through the components of $(-\varepsilon, \varepsilon) \times \partial W \setminus S(f)$ on which $f$ is an orientation preserving immersion.
Now, $Z_{\sigma}(f, \varphi)$ is obtained from the finite set $P_{f}$ (considered as a $0$-dimensional CW-complex) by attaching for every member $\{x, y\}$ of the partitions of $P_{f}$ induced by $\varphi$ and $\pi_{\sigma}(f)$ a $1$-cell with endpoints $x$ and $y$.

The quantity $c_{\sigma}(f, \varphi)$ captures the combinatorial interplay between $\varphi$ and $S(f)$ for given $\sigma$.
\end{itemize}

By construction we have the following

\begin{proposition}\label{proposition surgery handle attachment}
Let $(f, \varphi)$ be a pattern on $W$, and let $F$ be a realization of $(f, \varphi)$ without cusps.
Then, the number of boundary components of $W_{F}^{\sigma}$ (see \Cref{definition orientation preserving part of W}) equals $c_{\sigma}(f, \varphi) + l_{F} + n_{\sigma}$, where $l_{F}$ is the number of loops of $F$.
Moreover, there exists an integer $h_{F}^{\sigma} \geq 0$ such that
$$
\chi(W_{F}^{\sigma}) = c_{\sigma}(f, \varphi) + l_{F} + n_{\sigma} - 2 h_{F}^{\sigma}.
$$

Furthermore, if $P_{f} \neq \emptyset$, then there exists a closed one-dimensional submanifold $Z \subset W_{F}^{\sigma} \setminus \partial W_{F}^{\sigma}$ such that the number of components of $Z$ is $h_{F}^{\sigma} - n_{\sigma} \; (\geq 0)$, and such that every component of $W_{F}^{\sigma} \setminus Z$ has nonempty intersection with $S(F)$.
\end{proposition}

\begin{proof}
It follows from the definitions of $c_{\sigma}(f, \varphi)$ and $n_{\sigma}$ that the number of boundary components of $W_{F}^{\sigma}$ is $c_{\sigma}(f, \varphi) + l_{F} + n_{\sigma}$.
To get the formula for $\chi(W_{F}^{\sigma})$, choose $h_{F}^{\sigma}$ to be the number of surgeries on embedded $0$-spheres $S^{0}$ that are necessary to obtain $W_{F}^{\sigma}$ from the disjoint union of $c_{\sigma}(f, \varphi) + l_{F} + n_{\sigma}$ copies of $D^{2}$, the closed unit disc in $\mathbb{R}^{2}$.

It remains to construct the desired submanifold $Z \subset W_{F}^{\sigma} \setminus \partial W_{F}^{\sigma}$ under the assumption that $P_{f} \neq \emptyset$.
For this purpose, note that $c_{\sigma}(f, \varphi) + l_{F}$ is the number of boundary components of $W_{F}^{\sigma}$ which have nonempty intersection with $S(F)$.
Let $X$ and $Y$ denote the disjoint unions of $c_{\sigma}(f, \varphi) + l_{F}$ and $n_{\sigma}$ copies of $D^{2}$, respectively.
Recall from the construction of $h_{F}^{\sigma}$ that there exist pairwise disjoint embeddings
$$
\varphi_{i} \colon S^{0} \hookrightarrow (X \setminus \partial X) \sqcup (Y \setminus \partial Y), \qquad i = 1, \dots, h_{F}^{\sigma},
$$
such that $W_{F}^{\sigma}$ can be obtained from the disjoint union $X \sqcup Y$ by means of surgeries on the embeddings $\varphi_{i}$, and $\partial X$ becomes the union of those boundary components of $W_{F}^{\sigma}$ that have nonempty intersection with $S(F)$.
Since $W$ is connected and $P_{f} \neq \emptyset$ by assumption, every component of $W_{F}^{\sigma}$ must contain at least one boundary component which has nonempty intersection with $S(F)$.
Thus, by moving some of the handles $S^{1} \times [-1, 1] \subset W_{F}^{\sigma}$ associated with the embeddings $\varphi_{i}$ over each other, we may achieve that for every component $D$ of $Y$, there is an index $i_{D}$ such that the embedding $\varphi_{i_{D}}$ maps one point of $S^{0}$ to $X$ and the other point to $D$.
In particular, $h_{F}^{\sigma} - n_{\sigma} \geq 0$, and we may take $Z$ to be the union of the circles $S^{1} \times \{0\} \subset S^{1} \times [-1, 1]$ in the handles associated with the remaining $h_{F}^{\sigma} - n_{\sigma}$ embeddings.
\end{proof}

The statement of our main result \Cref{MAIN THEOREM 2} involves the quantity
\begin{displaymath}
\Delta^{\sigma} := \Gamma^{\sigma}/2 - c_{\sigma}(f, \varphi) + n_{\sigma}(f),
\end{displaymath}
which also depends on $\Gamma^{\sigma}$ as introduced in \Cref{proposition parity of number of cusps of realizations}.
Note that \Cref{proposition boundary turning invariant and Euler characteristic} and \Cref{proposition surgery handle attachment} together imply

\begin{corollary}\label{corollary computation of delta invariant in terms of data of a realization without cusps}
If $F$ is a realization of a pattern $(f, \varphi)$ of $W$ without cusps, then
$$\Delta^{\sigma} = l_{F} - 2 h_{F}^{\sigma} + 2 n_{\sigma}.$$
\end{corollary}

\section{Proof of \Cref{MAIN THEOREM 2}}\label{proof of main theorem 2}

First suppose that $W$ is non-orientable.
In view of loop generation (\Cref{lemma new fold loops}(i)) and loop reduction (\Cref{proposition loop reduction}) it suffices to show that $F$ can be modified on $W \setminus \partial W$ by a finite number of (E) and (C) moves to obtain a realization $\widetilde{F}$ of $(f, \varphi)$ without cusps in such a way that the number of loops of $F$ does not have the same parity as the number of loops of $\widetilde{F}$.
For this purpose, let $C \subset W \setminus \partial W$ be an embedded circle which intersects $S(F)$ transversely, say in the points $p_{1}, \dots, p_{r}$.
We modify $F$ by applying (C) once to create a pair of cusps at $p_{1}$.
Next, we modify the resulting realization $F_{1}$ of $(f, \varphi)$ by applying general elimination of cusps (\Cref{lemma generalized cusp elimination}) to a curve $\alpha$ that connects the two new cusps by going once around in a tubular neighborhood of $C$ while intersecting $S(F_{1})$ in $r-1$ fold points.
Note that the total number of (E) moves involved in the modification of $F_{1}$ equals $r$.
As a result, we obtain an extension $F_{2}$ of $f$ without cusps.
However, the (E) moves we used to modify $F_{1}$ may have changed the singular pattern, so that $F_{2}$ is not necessarily a realization of the pattern $(f, \varphi)$.
Nevertheless, we can modify $F_{2}$ further to obtain the desired realization $\widetilde{F}$ of the pattern $(f, \varphi)$ as follows.
Every (E) move we have used can be considered as the first step of loop generation (\Cref{lemma new fold loops}(ii)), and we perform the remaining moves.
All in all, we obtain a realization $\widetilde{F}$ of the pattern $(f, \varphi)$ without cusps, but the number of loops has increased by $r$ due to $r$-fold application of loop generation.
Hence, our claim follows from the fact that $r$ will be odd whenever $C$ is chosen such that its normal line bundle is non-trivial.
This completes the proof of part (a).
\par\medskip

As for the proof of part (b), suppose that $W$ is orientable.
Let $\sigma$ denote an orientation of $\partial W$ which is induced by one of the two orientations of $W$.
Moreover, let $P_{f} \neq \emptyset$.

$(i) \Rightarrow (ii)$.
Let $F$ be a realization of the pattern $(f, \varphi)$ which has no cusps and $l_{F}$ loops.
Then, \Cref{corollary computation of delta invariant in terms of data of a realization without cusps} implies that $\Delta^{\rho} \equiv l_{F} \, \operatorname{mod} 2$ and $\Delta^{\rho} = l_{F} - 2(h_{F}^{\rho} - n_{\rho}) \leq l_{F}$ for $\rho \in \{-\sigma, \sigma\}$, where $h_{F}^{\rho} \geq n_{\rho}$ by \Cref{proposition surgery handle attachment}.
Consequently, $l_{F} \in \mathbb{N} \cap (\Delta^{\sigma} + 2 \mathbb{N}) \cap (\Delta^{-\sigma} + 2 \mathbb{N})$.

$(ii) \Rightarrow (i)$.
Let $l \in \mathbb{N} \cap (\Delta^{\sigma} + 2 \mathbb{N}) \cap (\Delta^{-\sigma} + 2 \mathbb{N})$.
In view of loop generation (\Cref{lemma new fold loops}(i)) it suffices to show that $F$ can be modified on $W \setminus \partial W$ by a finite sequence of (E) and (C) moves to a realization of $(f, \varphi)$ which has no cusps and $\operatorname{min} (\mathbb{N} \cap (\Delta^{\sigma} + 2 \mathbb{N}) \cap (\Delta^{-\sigma} + 2 \mathbb{N}))$ loops.
This can be achieved as follows.
Since $P_{f} \neq \emptyset$, we can use loop simplification (\Cref{proposition loop simplification}) to achieve that all loops of $F$ are contractible.
Using tunneling (\Cref{proposition tunneling}) and loop reduction (\Cref{proposition loop reduction}), one can in addition achieve that the remaining loops of $F$ are boundary components of the same component of $W \setminus S(F)$, say a component $V$ of $W_{F}^{\rho}$ for suitable orientation $\rho \in \{\pm\sigma\}$.
Since all loops of $F$ are contractible, it follows that every loop bounds a contractible component of $W_{F}^{-\rho}$.
In this situation, an argument analogous to the last part of the proof of \Cref{proposition surgery handle attachment} implies that
$$
l_{F} + n_{\rho} \leq h_{F}^{\rho}.
$$
(In fact, note that $c_{\rho}(f, \varphi)$ is an upper bound for the number of components of $W_{F}^{\rho}$ because all loops of $F$ lie in the boundary of $V$.)
Hence, by \Cref{corollary computation of delta invariant in terms of data of a realization without cusps},
$$
\Delta^{\rho} = l_{F} - 2 h_{F}^{\rho} +2 n_{\rho} \leq 2(l_{F} - h_{F}^{\rho} + n_{\rho}) \leq 0.
$$
Therefore, using $\Delta^{\rho} \equiv l_{F} \equiv \Delta^{-\rho} \, \operatorname{mod} 2$, we obtain
$$
\operatorname{min} (\mathbb{N} \cap (\Delta^{\sigma} + 2 \mathbb{N}) \cap (\Delta^{-\sigma} + 2 \mathbb{N})) = \operatorname{min} (\mathbb{N} \cap (\Delta^{-\rho} + 2 \mathbb{N})).
$$
Finally, we modify $F$ by means of balancing (\Cref{proposition balancing}) to reduce the number of loops to $\operatorname{min} (\mathbb{N} \cap (\Delta^{-\rho} + 2 \mathbb{N}))$ as follows.
By \Cref{proposition surgery handle attachment}, there is a closed one-dimensional submanifold $Z \subset W_{F}^{-\rho} \setminus \partial W_{F}^{-\rho}$ that consists of $b := h_{F}^{-\rho} - n_{-\rho} \geq 0$ components, and such that every component of $W_{F}^{-\rho} \setminus Z$ has nonempty intersection with $S(F)$.
For every component $H$ of $Z$ we may then choose an embedding $\alpha_{H} \colon [0, 1] \rightarrow W \setminus \partial W$ with the following properties:
\begin{itemize}
\item The endpoints of $\alpha_{H}$ satisfy $\alpha_{H}(0), \alpha_{H}(1) \in W_{F}^{\rho} \setminus (\partial W \cup S(F))$.
\item $\alpha_{H}$ intersects $S(F)$ transversely, and $\alpha_{H}^{-1}(S(F)) = \{t_{0}^{H}, t_{1}^{H}\}$ for some $0 < t_{0}^{H} < t_{1}^{H} < 1$.
\item $\alpha_{H}$ intersects $H$ transversely, and $\alpha_{H}^{-1}(H) = \{s^{H}\}$ for some $s^{H} \in (t_{0}^{H}, t_{1}^{H})$.
\item $\alpha_{H}(t_{0}^{H})$ and $\alpha_{H}(t_{1}^{H})$ are fold points of $F$ that do not lie on the loops of $F$.
(Note that the components of $W_{F}^{-\rho}$ that contain a loop of $F$ in the boundary are disjoint to $H$ because the loops of $F$ are all contractible.)
\item If $H$ and $H'$ are different components of $Z$, then $\alpha_{H}([0, 1]) \cap \alpha_{H'}([0, 1]) = \emptyset$.
\end{itemize}
If $\Delta^{-\rho} > 0$, i.e, $l_{F} > 2(h_{F}^{-\rho} - n_{-\rho}) = 2b$, then we perform $b \; (\geq 0)$ balancing moves to reduce the number of loops to $\Delta^{-\rho} = \operatorname{min} (\mathbb{N} \cap (\Delta^{-\rho} + 2 \mathbb{N}))$ as required.
(More precisely, for every component $H$ of $Z$ we first apply tunneling (\Cref{proposition tunneling}) to move an unused pair $(C_{0}^{H}, C_{1}^{H})$ of loops of $F$ to the endpoints of $\alpha_{H}$, that is, $\alpha_{H}^{-1}(C_{i}^{H}) = \{i\}$ for $i = 0, 1$.
This is always possible since $W$ is connected, and all loops of $F$ are contractible, and hence trivial.
The desired balancing moves can then be performed because the assumptions of \Cref{proposition balancing} are satisfied for our choices of $H$ and $\alpha_{H}$.)
If, however, $\Delta^{-\rho} \leq 0$, i.e, $l_{F} \leq 2(h_{F}^{-\rho} - n_{-\rho}) = 2 b$, then let $\varepsilon \in \{0, 1\}$ such that $\varepsilon \equiv l_{F} \operatorname{mod} 2$.
Then, we can analogously perform $(l_{F}-\varepsilon)/2$ balancing moves to reduce the number of loops to $\varepsilon = \operatorname{min} (\mathbb{N} \cap (\Delta^{-\rho} + 2 \mathbb{N}))$.

This completes the proof of \Cref{MAIN THEOREM 2}.

\begin{remark}\label{remark apparent contour}
Our approach ignores the complexity of the \emph{apparent contour} $F(S(F))$ of a generic map $F \colon W \rightarrow \mathbb{R}^{2}$.
In particular, when $\partial W = \emptyset$ and $F$ is \emph{stable} (i.e., the immersion $F| \colon S(F) \rightarrow \mathbb{R}^{2}$ self-intersects only at fold points of $F$, and all self-intersections are nodes, that is, transverse double points), results including the number of nodes of $F$ can for instance be found in \cite{tyama}.
It might be interesting to think about a version of \Cref{MAIN THEOREM 2} for stable maps of surfaces with boundary into the plane (compare \cite{bruce}) that also includes the number of nodes.
Note that the number of nodes cannot directly be controlled by our algorithm since the (E) and (C) moves easily produce new nodes.
\end{remark}

\section{Applications}\label{applications}
Throughout this section we suppose $W$ to be orientable (the non-orientable case could be handled similarly).
Let $\sigma$ denote an orientation of $P = \partial W$ which is induced by one of the two orientations of $W$.

In this section the spirit of the gluing principle is reflected in our applications of \Cref{MAIN THEOREM 2} to pseudo-immersions (see \Cref{corollary extending main result to pseudo-immersions} and \Cref{example pseudo-immersions}) and cusp counting (see \Cref{corollary extending main result to cusps}).

\begin{proposition}\label{corollary extending main result to pseudo-immersions}
Let $F$ be a realization of the pattern $(f, \varphi)$ without cusps.
Then, the following statements are equivalent for any integer $l \geq 1$:
\begin{enumerate}[(i)]
\item $F$ can be modified on $W \setminus \partial W$ by a finite sequence of (E) and (C) moves to a realization of $(f, \varphi)$ which has no cusps and $l$ loops.
\item $l \in 1 + \mathbb{N} \cap (\Delta^{\sigma} -1 + 2 \mathbb{N}) \cap (\Delta^{-\sigma} -1 + 2 \mathbb{N})$.
\end{enumerate}
\end{proposition}

\begin{proof}
We present the proof of implication $(i) \Rightarrow (ii)$ in detail -- the proof of the converse implication is similar.

Let $G$ be a realization of $(f, \varphi)$ which has no cusps and $l$ loops.
Let $C$ be the boundary of a small open disc $U$ centered at a fold point on a loop of $G$ such that $C$ intersects $S(G)$ transversely in precisely two points.
Then, the surface $V := W \setminus U$ has boundary $\partial V = P \sqcup C$, which we equip with an orientation $\rho$ such that $\rho|_{\partial W} = \sigma$, and such that $\rho$ is induced by an orientation of $V$.
(If $\partial W \neq \emptyset$, then $\rho$ is uniquely determined by $\sigma$.)
Now $H := G|_{V}$ is a fold map on $V$ with $l-1$ loops.
If we equip $C$ with a suitable collar neighborhood in $V$, then $H$ can be considered as a realization of the pattern $(f \sqcup f_{\operatorname{fold}}, \varphi \sqcup \varphi_{2})$, where $f_{\operatorname{fold}}$ is taken from \Cref{example boundary turning invariant of cusps and fold points}, and $\varphi_{2}$ denotes the unique fixed-point free involution on $C \cap S(G)$.
Hence, using $\chi(W) = \chi(V) + 1$, $|(P \sqcup C)_{f \sqcup f_{\operatorname{fold}}}| = |P_{f}|+2$, and $\omega_{\pm\rho}(f \sqcup f_{\operatorname{fold}}) = \omega_{\pm\sigma}(f)$ (see \Cref{example boundary turning invariant of cusps and fold points}), we obtain
\begin{displaymath}
\Gamma^{\pm\rho}_{V} := \chi(V) + |(P \sqcup C)_{f \sqcup f_{\operatorname{fold}}}|/2 + \omega_{\pm\rho}(f \sqcup f_{\operatorname{fold}}) = \Gamma^{\pm\sigma}.
\end{displaymath}
Moreover, $n_{\pm\rho}(f \sqcup f_{\operatorname{fold}}) = n_{\pm\sigma}(f)$ and $c_{\pm\rho}(f \sqcup f_{\operatorname{fold}}, \varphi \sqcup \varphi_{2}) = c_{\pm\sigma}(f, \varphi) + 1$ imply
\begin{displaymath}
\Delta^{\pm\rho}_{V} := \Gamma^{\pm\rho}_{V}/2 - c_{\pm\rho}(f \sqcup f_{\operatorname{fold}}, \varphi \sqcup \varphi_{2}) + n_{\pm\rho }(f \sqcup f_{\operatorname{fold}}) = \Delta^{\pm\sigma} - 1.
\end{displaymath}
Finally, the implication $(i) \Rightarrow (ii)$ of \Cref{MAIN THEOREM 2}(b) implies that $l-1 \in \mathbb{N} \cap (\Delta^{\sigma} -1 + 2 \mathbb{N}) \cap (\Delta^{-\sigma} -1 + 2 \mathbb{N})$.
\end{proof}

\begin{example}[Pseudo-immersions]\label{example pseudo-immersions}
An important special case of \Cref{corollary extending main result to pseudo-immersions} arises when $f$ restricts to an immersion $(-\varepsilon', \varepsilon') \times \partial W \rightarrow \mathbb{R}^{2}$ for some $\varepsilon' \in (0, \varepsilon)$.
In this case, realizations $F$ of the pattern $(f, \varphi_{\emptyset})$ (where $\varphi_{\emptyset}$ denotes the unique permutation of the empty set $\emptyset$) are called \emph{pseudo-immersions} (note that $S(F) \neq \emptyset$ by \Cref{remark singular locus not empty}).
The case that $\partial W$ is connected has been considered in \cite[Theorems 1.2 and 1.3]{myam}, and \Cref{corollary extending main result to pseudo-immersions} can be reduced to it as follows.

The assumption on $f$ implies that $P_{f} = \emptyset$ and $\omega_{\sigma}(f) = W(F|_{\partial W})$, the winding number (see \Cref{boundary turning invariant}).
Hence, $\Gamma^{\sigma} = \chi(W) + W(F|_{\partial W})$ by \Cref{proposition parity of number of cusps of realizations}.
Moreover, $n_{\sigma}(f) = 1$ (for suitable $\sigma$), and $c_{\sigma}(f, \varphi_{\emptyset}) = 0$.
Thus, we obtain
\begin{displaymath}
\Delta^{\sigma} = (\chi(W) + W(F|_{\partial W}))/2 + 1 = ((\chi(W) + 1) + (W(F|_{\partial W})+1))/2.
\end{displaymath}
On the other hand, observe that $\omega_{-\sigma}(f) = - \omega_{\sigma}(f) = -W(F|_{\partial W})$, $n_{-\sigma}(f) = 0$, and $c_{-\sigma}(f, \varphi_{\emptyset}) = 0$, which yields
\begin{displaymath}
\Delta^{-\sigma} = (\chi(W) - W(F|_{\partial W}))/2 = ((\chi(W) + 1) - (W(F|_{\partial W})+1))/2.
\end{displaymath}
Altogether, $\mathfrak{m} := \operatorname{max}\{\Delta^{\sigma}, \Delta^{-\sigma}\} = (\chi(W) + 1 + |W(F|_{\partial W})+1|)/2$, and by \Cref{corollary extending main result to pseudo-immersions}, in accordance with \cite{myam}, the set of possible numbers of loops of $F$ is
\begin{align*}
1 + \mathbb{N} \cap (\mathfrak{m} -1 + 2 \mathbb{N}) =
\begin{cases}
\operatorname{max}(\mathfrak{m}, 2) + 2 \mathbb{N}, \quad \text{if } \chi(W) - W(F|_{\partial W}) \equiv 0 \, \operatorname{mod} 4, \\
\operatorname{max}(\mathfrak{m}, 1) + 2 \mathbb{N}, \quad \text{if } \chi(W) - W(F|_{\partial W}) \equiv 2 \, \operatorname{mod} 4.
\end{cases}
\end{align*}
\end{example}

\begin{remark}
Note that the proof of \cite[Theorem 1.3]{myam} relies on a theorem due to Eliashberg \cite{eli} and Francis \cite{fra} (see \cite[Theorem 3.2, p. 1330]{myam}).
On the other hand, our approach is purely based on (E) and (C) moves, and we can easily deduce the Eliashberg-Francis theorem from \Cref{corollary extending main result to pseudo-immersions}.
\end{remark}

\begin{proposition}\label{corollary extending main result to cusps}
Let $F$ be a realization of the pattern $(f, \varphi)$.
If $P_{f} \neq \emptyset$, then the following statements are equivalent for any integers $c > 0$ and $l$:
\begin{enumerate}[(i)]
\item $F$ can be modified on $W \setminus \partial W$ by a finite sequence of (E) and (C) moves to a realization of $(f, \varphi)$ which has $c$ cusps and $l$ loops.
\item
$l \in \bigcup_{w \in \{-c, -c+2, \dots, c-2, c\}} \mathbb{N} \cap (\Delta^{\sigma} +w/2 + 2 \mathbb{N}) \cap (\Delta^{-\sigma} - w/2 + 2 \mathbb{N})$.
\end{enumerate}
\end{proposition}

\begin{proof}
The proof is based on similar ideas as the proof of \Cref{corollary extending main result to pseudo-immersions}.

$(i) \Rightarrow (ii)$.
Let $G$ be a realization of $(f, \varphi)$ which has $c \geq 1$ cusps and $l$ loops.
Let $U \subset W$ be the union of $c$ small open discs each of which is centered at a different cusp of $G$.
By choosing those small discs appropriately, we may assume that the boundary $C$ of $U$ is diffeomorphic to the disjoint union of $c$ circles each of which intersects $S(G)$ transversely in precisely two points.
Then, the surface $V := W \setminus U$ has boundary $\partial V = P \sqcup C$, which we equip with the unique orientation $\rho$ such that $\rho|_{\partial W} = \sigma$ (where note that $\partial W \neq \emptyset$), and such that $\rho$ is induced by an orientation of $V$.
Now $H := G|_{V}$ is a fold map on $V$ with $l - d$ loops, where $d$ denotes the number of loops of $G$ that contain at least one cusp of $G$.
If we equip $C$ with a suitable collar neighborhood in $V$, then $H$ is near every component of $C$ an extension of the boundary condition $f_{\operatorname{cusp}}$ considered in \Cref{example boundary turning invariant of cusps and fold points}.
Thus, $H$ can be considered as a realization of the pattern $(f \sqcup g, \psi)$, where $g$ denotes the disjoint union of $c$ copies of $f_{\operatorname{cusp}}$, and $\psi$ denotes the partition of $\partial V \cap S(G)$ into those subsets of cardinality $2$ which arise as the boundary points of some component of $S(H)$.
Hence, using $\chi(W) = \chi(V) + c$, $|(P \sqcup C)_{f \sqcup g}| = |P_{f}|+2c$, and $\omega_{\pm\rho}(f \sqcup g) = \omega_{\pm\sigma}(f) \pm w$ for some integer $w$ with the same parity as $c$, and such that $|w| \leq c$, we obtain
\begin{displaymath}
\Gamma^{\pm\rho}_{V} := \chi(V) + |(P \sqcup C)_{f \sqcup g}|/2 + \omega_{\pm\rho}(f \sqcup g) = \Gamma^{\pm\sigma} \pm w.
\end{displaymath}
Moreover, $n_{\pm\rho}(f \sqcup g) = n_{\pm\sigma}(f)$ and $c_{\pm\rho}(f \sqcup g, \psi) = c_{\pm\sigma}(f, \varphi) + d$ imply
\begin{displaymath}
\Delta^{\pm\rho}_{V} := \Gamma^{\pm\rho}_{V}/2 - c_{\pm\rho}(f \sqcup g, \psi)+ n_{\pm\rho}(f \sqcup g) = \Delta^{\pm\sigma} \pm w/2 - d.
\end{displaymath}
Finally, to conclude that $(ii)$ holds, we apply the implication $(i) \Rightarrow (ii)$ of \Cref{MAIN THEOREM 2}(b) to the realization $H$ of the pattern $(f \sqcup g, \psi)$ to obtain
$$
l-d \in \mathbb{N} \cap (\Delta^{\sigma} + w/2 - d + 2 \mathbb{N}) \cap (\Delta^{-\sigma} - w/2 - d + 2 \mathbb{N}).
$$

$(ii) \Rightarrow (i)$.
We may choose an integer $w \in \{-c, -c+2, \dots, c-2, c\}$ such that
$$
l \in \mathbb{N} \cap (\Delta^{\sigma} +w/2 + 2 \mathbb{N}) \cap (\Delta^{-\sigma} - w/2 + 2 \mathbb{N}).
$$
We modify $F$ as follows.
After possibly creating a number of new pairs of cusps for $F$ by means of (C), we can choose cusps $x_{1}, \dots, x_{c}$ of $F$ admitting pairwise disjoint small open disc neighborhoods $U_{1}, \dots, U_{c}$, respectively, where the boundary circle $C_{i}$ of each $U_{i}$ intersects $S(F)$ transversely in precisely two points, and the following property holds.
If $f_{i}$ denotes the restriction of $F$ to some fixed collar neighborhood of $C_{i}$ in $V := W \setminus \bigsqcup_{i} U_{i}$, and $\rho$ denotes the unique orientation of $\partial V = \partial W \sqcup \bigsqcup_{i} C_{i}$ that is induced by an orientation of $V$ in such a way that $\rho|_{\partial W} = \sigma$ (where note that $\partial W \neq \emptyset$), then
$$
w = \sum_{i = 1}^{c} \omega_{\rho|_{C_{i}}}(f_{i}).
$$
(Indeed, since $w \in \{-c, -c+2, \dots, c-2, c\}$ can clearly be written as the sum of $c$ summands of the form $\pm 1$, we only have to apply the following observation.

If we create a new pair $(x, x')$ of cusps for $F$ by means of (C) and choose sufficiently small open disc neighborhoods $U$ and $U'$ of $x$ and $x'$, respectively, whose respective boundary circles $C$ and $C'$ intersect $S(F)$ transversely in precisely two points, then the following property holds.
If $g$ and $g'$ denote the restrictions of $F$ to some fixed collar neighborhoods of $C$ and $C'$ in $Y = W \setminus (U \sqcup U')$, respectively, and $\tau$ denotes the unique orientation of $\partial Y = \partial W \sqcup C \sqcup C'$ that is induced by an orientation of $Y$ in such a way that $\tau|_{\partial W} = \sigma$, then $\omega_{\tau|_{C}}(g), \omega_{\tau|_{C'}}(g') \in \{\pm 1\}$ and $\omega_{\tau|_{C}}(g) = - \omega_{\tau|_{C'}}(g')$.

In fact, the statement $\omega_{\tau|_{C}}(g), \omega_{\tau|_{C'}}(g') \in \{\pm 1\}$ follows because we can choose the neighborhoods $U$ and $U'$ so that $C$ and $C'$ have collar neighborhoods in $Y = W \setminus (U \sqcup U')$, respectively, in which $g$ and $g'$ look like copies of the boundary condition $f_{\operatorname{cusp}}$ considered in \Cref{example boundary turning invariant of cusps and fold points}.
Moreover, the statement $\omega_{\tau|_{C}}(g) = - \omega_{\tau|_{C'}}(g')$ can be deduced as follows.
Since the pair $(x, x')$ of cusps has been created by means of (C), we can choose the closures of $U$ and $U'$ to lie in a small closed disc $Z \subset W \setminus \partial W$ such that the boundary circle $\partial Z$ intersects $S(F)$ transversely in precisely two points, and the restriction of $F$ to some fixed collar neighborhood of $\partial Z$ in $W \setminus Z$ looks like the boundary condition $f_{\operatorname{fold}}$ considered in \Cref{example boundary turning invariant of cusps and fold points}.
Then the claim follows by applying \Cref{proposition boundary turning invariant and Euler characteristic} to $F|_{Z \setminus (U \sqcup U')}$.)

The restriction $H := F|_{V}$ is a generic map on $V$ that can be considered as a realization of a pattern of the form $(h, \psi)$, where $h = f \sqcup \bigsqcup_{i}f_{i}$, and where $\psi$ denotes the partition of $\partial V \cap S(G)$ into those subsets of cardinality $2$ which arise as the boundary points of some component of $S(H)$.
Hence, using $\chi(W) = \chi(V) + c$, $|(\partial V)_{h}| = |P_{f}|+2c$, as well as $\omega_{\pm\rho}(h) = \omega_{\pm\sigma}(f) \pm w$, we obtain
\begin{displaymath}
\Gamma^{\pm\rho}_{V} := \chi(V) + |(\partial V)_{h}|/2 + \omega_{\pm\rho}(h) = \Gamma^{\pm\sigma} \pm w.
\end{displaymath}
Without loss of generality, we may assume that the cusps $x_{1}, \dots, x_{c}$ lie on components of $S(F)$ that have nonempty intersection with $\partial W$.
Consequently, $c_{\pm\rho}(h, \psi) = c_{\pm\sigma}(f, \varphi)$.
Using also that $n_{\pm\rho}(h) = n_{\pm\sigma}(f)$, we obtain
\begin{displaymath}
\Delta^{\pm\rho}_{V} := \Gamma^{\pm\rho}_{V}/2 - c_{\pm\rho}(h, \psi) + n_{\pm\rho}(h) = \Delta^{\pm\sigma} \pm w/2.
\end{displaymath}

Note that $\Gamma^{\rho}_{V}$ is even because $\Delta^{\rho}_{V} = \Delta^{\sigma} + w/2$ is an integer by assumption.
Hence, \Cref{proposition parity of number of cusps of realizations} implies that $H$ has an even number of cusps.
Thus, according to \Cref{MAIN THEOREM 1}(b), $H$ can be modified on $V \setminus \partial V$ by a finite sequence of (E) and (C) moves to obtain a realization $H_{0}$ of $(h, \psi)$ without cusps.
Since
$$
l \in \mathbb{N} \cap (\Delta^{\rho}_{V} + 2 \mathbb{N}) \cap (\Delta^{-\rho}_{V} + 2 \mathbb{N}),
$$
we may apply the implication $(ii) \Rightarrow (i)$ of \Cref{MAIN THEOREM 2}(b) to modify the realization $H_{0}$ of the pattern $(h, \psi)$ by a finite sequence of (E) and (C) moves to a realization of $(h, \psi)$ which has no cusps and $l$ loops.
The above modifications of $H$ and $H_{0}$ can be understood as modifications of $F$ on $W \setminus \partial W$ by a finite sequence of (E) and (C) moves that turn $F$ into a realization of $(f, \varphi)$ which has $c$ cusps and $l$ loops.
\end{proof}

\begin{remark}\label{remark positive TFT}
The work on this paper has been motivated by Banagl's recent construction of TFT-type invariants \cite{ban}.
In fact, Banagl has introduced the notion of a positive TFT as a convenient framework for constructing TFT-type invariants based on \emph{semirings} rather than on rings.
Following the essential feature of a TFT, the state sum (or partition function) of a positive TFT is required to satisfy a gluing axiom, which forces the information it detects to be local to a certain extent.
As for the technical implementation of Banagl's positive TFT based on fold maps, the combinatorial information captured by singular patterns of fold maps is incorporated into the morphisms of a category.
The resulting \emph{Brauer category} is a strict monoidal category that is constructed as categorification of the Brauer algebras classically known from representation theory of the orthogonal group $O(n)$.
By means of \Cref{MAIN THEOREM 2} we are able to list those integers $l$ for which a given singular pattern on $W$ admits a realization having no cusps and $l$ loops.
From the perspective of Banagl's positive TFT based on fold maps, this leads to the computation of the state sum invariant on all $2$-dimensional cobordisms (see \cite{wra2} for details).
\end{remark}

\bibliographystyle{amsplain}

\end{document}